%% file: hmm-ellam.tex
\documentclass[final]{amsart}

\usepackage{caption}
\usepackage{amsfonts,latexsym,graphicx,amsmath,amssymb}

\makeindex

\usepackage[color,notref,notcite]{showkeys}

\definecolor{labelkey}{rgb}{0.6,0,1}

\input{macros.tex}

\begin{document}

\title[An HMM--ELLAM scheme for flows in porous media]{An HMM--ELLAM scheme on generic polygonal meshes for miscible incompressible flows in porous media}
\author{Hanz Martin Cheng}
\address{School of Mathematical Sciences, Monash University, Clayton, Victoria 3800, Australia.
\texttt{hanz.cheng@monash.edu}}
\author{J\'er\^ome Droniou}
\address{School of Mathematical Sciences, Monash University, Clayton, Victoria 3800, Australia.
\texttt{jerome.droniou@monash.edu}}
\date{\today}

%
%

\maketitle




\begin{abstract}
We design a numerical approximation of a system of partial differential equations modelling
the miscible displacement of a fluid by another in a porous medium. The advective part of the system
is discretised using a characteristic method, and the diffusive parts by a finite volume
method. The scheme is applicable on generic (possibly non-conforming) meshes as encountered in applications.
The main features of our work are the reconstruction of a Darcy velocity, from the discrete pressure fluxes,
that enjoys a local consistency property, an analysis of implementation issues faced when tracking, via the characteristic method, distorted cells, and a new treatment of cells near the injection well that accounts better for the conservativity of the injected fluid.
\end{abstract}

\section{Introduction}
This study focuses on the recovery of oil in a process known as miscible displacement, in which a solvent, such as a short-chain hydrocarbon or pressurised carbon-dioxide, is injected into the oil reservoir to reduce the viscosity of the resident oil and push it towards a production well. One of the models that describes the said process is the Peaceman model, which was first introduced in \cite{PR62}.
\newline

Let $\Omega$ be a bounded domain in $\mathbb{R}^{d}$ and $[0,T]$ be a time interval. Denote by $\mathbf{K}(\mathbf{x})$ and $\phi(\mathbf{x})$ the permeability tensor and the porosity of the medium, respectively.
Then, neglecting gravity, the Peaceman model is given by:
\begin{subequations}\label{eq:model}
	\begin{equation} \label{pressure}
	\begin{aligned}
	\nabla \cdot \mathbf{u} &= q^{+}-q^{-} := q\qquad \mbox{ on } \O \times [0,T] \\
	\mathbf{u} &= - \dfrac{\mathbf{K}}{\mu(c)} \nabla p \qquad \mbox{ on } \O \times [0,T] \\
	\end{aligned}
	\end{equation}
	\begin{equation} \label{concentration}
	\phi \dfrac{\partial c}{\partial t} + \nabla \cdot (\mathbf{u}c-\mathbf{D}(\mathbf{x},\mathbf{u})\nabla c) = q^{+}-cq^{-} := q_{c} \qquad \mbox{ on } \O \times [0,T] \\
	\end{equation}
	with unknowns  $p(\x,t), \darcyU(\x,t),$ and $c(\x,t)$ which denote the pressure of the mixture, the Darcy velocity, and the concentration of the injected solvent, respectively. Note that this form assumes an injection concentration
of $1$ (the trivial modification $q_c=\widehat{c}q^+-cq^-$ would allow for a generic injection concentration $\widehat{c}$, but numerical tests always consider $\widehat{c}=1$). We also note that the model \eqref{eq:model} is derived under the assumption that the fluid and the rock are incompressible. The more generic formulation for the pressure equation \eqref{pressure} is given by $\nabla\cdot\mathbf{u}=q + \phi c_s\dfrac{\partial p}{\partial t}$, where $c_s$ is the total compressibility of the system. The concentration equation \eqref{concentration} is then modified accordingly. In particular, the compressibility $c_s$ is related to $\dfrac{\partial \rho}{\partial p}$ where $\rho(p)$ is the density of the fluid, and the assumption that $c_s=0$ implies that the density is constant. However, in many cases, $c_s$ is very small and negligible, and hence the assumption on incompressibility is not very restrictive \cite{E83-Mathematics-Reservoir}.  As a matter of fact, this assumption is also used in some engineering applications \cite{CS17-incompressible-2phase-flow,P77-Reservoir-Simulation,W67-chemical}. In particular, for petroleum engineering, this may be used in a gas cap drive reservoir or when the reservoir pressure drops below the bubble-point pressure \cite[Chapter 7]{IHMMA16-advanced-petroleum-reservoir}. The model \eqref{eq:model} is understood in the following manner: \eqref{pressure} gives us the conservation of mass for the total fluid (mixture of oil and solvents), whereas \eqref{concentration} gives us the conservation of mass of the injected solvents. This captures the physical phenomenon of two or more components (oil and solvents) flowing along a single phase, for which each of the components have different concentration; the derivation and more interpretation of this model are given in \cite[Chapter 2]{E83-Mathematics-Reservoir}.
	
	The functions $q^{+}$ and $q^{-}$ represent the injection and production wells respectively, and $\mathbf{D}(\mathbf{x},\mathbf{u})$ denotes the diffusion tensor
	\begin{equation} \nonumber
	\mathbf{D}(\mathbf{x},\mathbf{u}) = \phi(\mathbf{x})\left[d_{m}\mathbf{I}+d_{l}|\mathbf{u}|\proj(\mathbf{u})+d_{t}|\mathbf{u}|\left(\mathbf{I}-\proj(\mathbf{u})\right)\right]
	\mbox{ with }\proj(\mathbf{u}) = \left(\dfrac{u_{i}u_{j}}{|\mathbf{u}|^{2}}\right)_{i,j}.
	\end{equation}
	Here, $d_{m}$ is the molecular diffusion coefficient, $d_{l}$ and $d_{t}$ are the longitudinal and transverse dispersion coefficients respectively, and $\proj(\mathbf{u})$ is the projection matrix along the direction of $\mathbf{u}$.
	Also, $\mu(c)=\mu(0)[(1-c)+M^{1/4}c]^{-4}$ is the viscosity of the fluid mixture, where $M=\mu(0)/\mu(1)$ is the mobility ratio of the two fluids.
	As usually considered in numerical tests, we consider no-flow boundary conditions, and
	zero initial conditions for the concentration:
	\begin{equation}
	\mathbf{u} \cdot \mathbf{n} = (\mathbf{D}\nabla c) \cdot \mathbf{n} = 0 
	\mbox{ on }\partial\O \times [0,T]\,,\qquad
	c(\cdot,0)=0\mbox{ in $\O$}.
	\end{equation}
\end{subequations}
The pressure is fixed by imposing a zero average: 
\begin{equation} \label{zeroAve}
\int_{\Omega}p(\x,t) \d\x= 0 \qquad \forall t\in [0,T]
\end{equation}

The purpose of this work is to design and test a numerical scheme for the complete coupled model \eqref{eq:model}.
This scheme is based on an adjoint method of characteristic for the advection in \eqref{concentration}
and a finite volume method for the diffusion terms. The main contributions of our work are:
\begin{itemize}
\item Usage of the Hybrid Mimetic Mixed (HMM) method for the diffusion terms, that is applicable
on meshes with very generic geometries as encountered in real-world applications.
\item Reconstruction of a velocity field from the numerical Darcy fluxes, that preserves
the exact divergence and is locally consistent on $\RT0$ functions; this reconstructed velocity is required
to apply the characteristic method.
\item A new treatment of cells near the injection well that accounts better for the conservativity of the injected fluid.
\end{itemize}

An early version of the HMM--ELLAM was presented in \cite{CD17-HMM-ELLAM}. This version however did not include the special treatment of cells near the injection well, or the new velocity reconstruction. We show here that the combination of these novel elements lead to improved numerical results compared with this initial version.

A number of various numerical methods have been used to approximate the solutions of \eqref{eq:model},
from finite difference techniques \cite{pe66,douglas83}, to finite element schemes \cite{dew83,ew80}, to
discontinuous Galerkin methods \cite{bjm09,rivwalk11}, to finite volume methods \cite{CD-07,ckm15}.
Closer to the focus of this article are the Modified Method of Characteristics (MMOC) and
the Eulerian-Lagrangian Localized Adjoint Method (ELLAM). These methods, designed to deal with the
advective terms, have been applied in conjunction with mixed finite elements (for the pressure equation) and
conforming finite elements (for the diffusion terms in the concentration equation) in \cite{cejs02a,erw84,WLELQ-00},
among others. The combination with FE methods severely restricts the cell geometries that can be managed
with such methods. On the contrary, recent finite volume (FV) methods can accommodate very generic
mesh geometries, see the review \cite{D14-FVschemes} and references therein. Among those, the
hybrid mimetic mixed (HMM) method of \cite{dro-10-uni} is a family of numerical schemes for diffusion equations,
applicable on generic meshes, which gathers three separately developed numerical methods:
hybrid finite volumes \cite{egh10}, mixed-hybrid mimetic finite differences \cite{bls05mimetic}, and mixed finite volume \cite{de06mixed}. The HMM has been adapted in \cite{CD-07} to the model \eqref{eq:model}, using an upwind discretisation of
the advective term. The drawback of such a discretisation is an additional numerical diffusion, which leads
to a widening of the transition layer between the regions $c\approx 1$ and $c\approx 0$.

\medskip

In this work, we propose to combine the HMM method, for its usability on generic grids, with
an ELLAM for the advection, for its reduced numerical diffusion compared to the upwinding
method. The HMM method being a FV scheme, using it to discretise the pressure equation \eqref{pressure}
naturally produces numerical fluxes of the Darcy velocity across the mesh faces. The ELLAM however
requires a complete velocity field, to track points along it (see equation \eqref{charac}). To avoid creating artificial
wells, this velocity field should have the exact same divergence as prescribed by the numerical
Darcy fluxes. To reconstruct such a divergence-preserving field starting from the Darcy fluxes,
we use ideas of \cite{KR03-kuznetsov-repin}: the cells are split in simplices (triangles in 2D)
and an $\RT0$ velocity field is reconstructed on these simplices such that its divergence on each
simplex is equal to the balance of the fluxes on the cell (discrete divergence of the Darcy velocity
in the cell). The corresponding system is under-determined and the original method of \cite{KR03-kuznetsov-repin}
suggests to select the solution with a certain minimal norm. We propose an alternative approach
by using, in the case of dimension 2, a complementary equation which is consistent on functions
that are globally $\RT0$ in the entire cell. In other words, if, on a given cell, the numerical Darcy fluxes
obtained by solving \eqref{pressure} are the fluxes of a field of the form $\mathbf{a}+b\x$ on the cell,
then the reconstructed velocity is exactly equal to that field in the cell. This construction
is performed in dimension $d=2$ here, but already shows promising results; the extension to the
three-dimensional case being the purpose of a future work.

As most low order FV methods, the HMM method approximates the concentration by piecewise constant
functions in the cells. The test functions are therefore also piecewise constant, and the ELLAM requires
to track them along the reconstructed velocity field. This boils down to tracking each cell along
the velocity field. This tracking is easier to perform than that of FE basis functions, which are intrinsically more complicated than piecewise constant functions; this is another advantage, in the ELLAM framework, of FV methods over FE methods.

Each cell is tracked by transporting its vertices and edge midpoints (and possibly more points along the edge, depending on the cell regularity) by the reconstructed Darcy velocity; the tracked cell is then approximated by the polygon formed by these tracked points. For cells tracked into the injection well, to mitigate the steepness of the source terms there, we use a different strategy, based on a trace forward algorithm
inspired by \cite{AW11-stability-monotonicity-implementation}. The original algorithm of \cite{AW11-stability-monotonicity-implementation} 
however sometimes lead to an excessive numerical injection of fluid. We therefore modify this specific treatment of
cells near the injection well to eliminate this excess. Aside from this, the integral of the source term should be treated properly; otherwise, mass conservation will not be satisfied. To this end, we use a weighted trapezoidal rule, as described in \cite{AH06}. 

\medskip

The paper is organised as follows.
The HMM--ELLAM scheme is presented in Section \ref{sec:scheme},
starting with the HMM scheme for the pressure equation \eqref{pressure}. We then proceed, in section \ref{numConc}, by presenting an ELLAM scheme for the advective component of concentration equation \eqref{concentration}, whilst still discretising the diffusive component using an HMM scheme. Discretisation of the source term should be treated carefully, and so it will be presented separately in section \ref{numSource}. In section \ref{KRDarcyU}, we present the reconstruction of the velocity, first by using the original method of \cite{KR03-kuznetsov-repin} involving the minimisation of some norm, then by presenting our approach based on a consistent complementary equation. Numerical results are presented, on various types of meshes, in section \ref{numResults}. These results show that the reconstructed velocity based on the consistent complementary equation behaves better, on distorted meshes, than the velocity chosen to have a minimal norm. We also provide qualitative and quantitative comparisons with the results obtained from MFV--upwind and MFEM--ELLAM schemes, outlining the positive and negative qualities  of the numerical solutions obtained from each scheme. A few comments about the proper implementation of ELLAM are also made in Section \ref{numResults}, such as the proper number of points to take for tracking each cell, and the impact of choosing the proper quadrature rule for integrating the source term. 

\section{The HMM--ELLAM}\label{sec:scheme}

The numerical solution of \eqref{eq:model} is usually approximated by using a two-step
process. Starting from a known value $c^{(n)}$ of $c$ at time level $n$ (for $n=0$, 
$c^{(0)}=0$), a numerical solution $p^{(n+1)}$ for $p$ at time level $n+1$ is computed by approximating \eqref{pressure} with $c=c^{(n)}$. This computation
also provides an approximation $\darcyU^{(n+1)}$ of the Darcy velocity at time level $n+1$, and possibly of secondary quantities (e.g., fluxes). The concentration $c^{(n+1)}$ at time level $n+1$ is then computed by approximating \eqref{concentration} by using $\darcyU=\darcyU^{(n+1)}$ and the aforementioned secondary quantities.

 For our discretisation, we will consider polytopal meshes as defined in \cite{DEGGH16}, in dimension $d=2$
(as explained above, this restriction is due to a special consistent reconstruction of the Darcy velocity from the numerical fluxes, whose extension to dimension 3 is the purpose of an ongoing work).
 Thus, $\mathcal T=(\mesh, \edges, \points)$ are the set of cells $K$, edges $\edge$, and points $\x_K$ (one per cell) of our mesh, respectively. The cells can
be any star-shaped polygons. For a cell $K\in\mesh$, $\edgescv\subset \edges$ denotes the set of edges of the cell $K\in\mesh$. We also denote by $|K|$ the area of a cell $K$, and $|\edge|$ the length of an edge $\edge$. As is usual, and not restrictive with respect to applications, we assume that the
permeability and porosity are constant in each cell, and we write $\mathbf{K}_K$ and $\phi_K$ their
respective values in the cell $K\in\mesh$.

\subsection{Numerical scheme for the pressure equation}
\label{numPress}

The variational formulation for \eqref{pressure} is given by 
	\begin{equation} \label{varpressure}
\int_{\Omega} \dfrac{\mathbf{K}}{\mu(c)}\grad p \cdot \grad \xi  = \int_{\Omega} q(t^{(n+1)})\xi		\,,\quad\forall \xi\in H^1(\O).
\end{equation}


There is an implicit time variable above, which is here just a parameter. The degrees of freedom of our numerical scheme are given by cell and edge values. The space of DOFs is then
\[
X_{\mathcal{T}} := \{w=((w_{\cv})_{\cv\in\mesh},(w_{\edge})_{\edge\in\edgescv}) \}.
\] 
For HMM schemes, a piecewise constant gradient is then defined on a sub-triangula\-tion of cells.
Following \cite{dro-10-uni}, if $K\in\mesh$ and $(T_{K,\edge})_{\edge\in\edgescv}$ are
the triangles defined by $\x_K$ and the edges of $K$ (see Figure \ref{fig.subd}), we set
\begin{equation}\label{def.grad} 
\begin{aligned}
&\forall w\in X_{\mathcal T}\,,\;\forall \x\in T_{K,\edge}\,,\\
&\grad_{H} w(\mathbf{x}) = \ograd_{\cv} w + \dfrac{\sqrt{2}}{\dcvedge}[w_{\edge}-w_{\cv}-\ograd_{\cv}w \cdot(\centeredge-\centercv)] \ncvedge\,,\\
&\mbox{ where }\ograd_{\cv}w=|\cv|^{-1}\sum_{\edge\in\edgescv} |\sigma|w_{\edge}\ncvedge.
\end{aligned}
\end{equation}
Note that $\ograd_\cv w$ is a linearly exact reconstruction of the gradient, that is, if $(w_{\edge})_{\edge\in\edgescv}$ interpolate an affine function $A$ at the edge midpoints, then $\ograd_\cv w=\nabla A$.
In \eqref{def.grad}, $\dcvedge$ is the orthogonal distance between $\centercv$ and $\edge$, and
the second term can be seen as a stabilisation of the gradient involving a discrete 2nd order Taylor expansion.

Given $p^{(n+1)}\in X_{\mathcal T}$,
the discrete Darcy fluxes $(F_{K,\edge}^{(n+1)}) _{K\in\mesh,\,\edge\in\edgescv}$,
approximations of $-\int_\edge (\mathbf{K}/\mu(c))\grad p \cdot \ncvedge$, are then defined by
\begin{equation}\label{def.fluxes.p}
\begin{aligned}
&\forall K\in\mesh\,,\;\forall v \in X_{\mathcal{T}}\,,\\
&\sum_{\edge \in \edgescv} F_{K,\sigma}^{(n+1)} (v_{K}-v_{\sigma}) = \int_{K} \dfrac{\mathbf{K}_K}{\mu(c^{(n)}_K)} \nabla_{H} p^{(n+1)}(\mathbf{x}) \cdot \nabla_{H} \xi(\mathbf{x}) d\mathbf{x}
\end{aligned}
\end{equation}
where $c^n_K$ is a cell-centred approximate concentration at the previous time step, which is
available since $c$ is also approximated in $X_{\mathcal T}$.

Fundamental in the formulation of finite volume schemes is the balance and conservativity of fluxes \cite{D14-FVschemes}, which are given by the equations
\begin{subequations}\label{scheme:pressure}
\begin{equation}\label{fluxBal}
\sum_{\edge\in\mathcal{E}_{K}} F^{(n+1)}_{\cv,\edge} = \int_{K} q(t^{(n+1)})\quad\mbox{ for all $K\in\mesh$},
\end{equation}
\begin{equation}\label{fluxCons}
\begin{aligned}
F^{(n+1)}_{K,\sigma}+F^{(n+1)}_{L,\sigma}=0 {}&\quad \mbox{ for all edges $\edge$ between two different cells $K$ and $L$,}\\
F^{(n+1)}_{K,\sigma}=0{}&\quad\mbox{ for all edges $\edge$ of $K$ lying on $\partial\O$}.
\end{aligned}
\end{equation}
As expected, the normalisation \eqref{zeroAve} of the pressure needs to be separately imposed:
\begin{equation}
\sum_{K\in\mesh}|K|p^{n+1}_K=0.
\end{equation}
\end{subequations}
With the definitions \eqref{def.grad}--\eqref{def.fluxes.p}, the system \eqref{scheme:pressure}
provides an approximation $p^{(n+1)}\in X_{\mathcal T}$ of $p$ at time $t^{(n+1)}$, as well as approximate Darcy fluxes $(F_{K,\edge}^{(n+1)})_{K\in\mesh\,,\;\edge\in\edgescv}$ at the same time.

\subsection{Numerical scheme for the concentration equation} \label{numConc}
For the concentration equation \eqref{concentration}, we have the following variational formulation:
\begin{equation}\label{conc.weak}
\int_{t^{(n)}}^{t^{(n+1)}}\!\!\!\!\int_{\Omega} \left(\phi \dfrac{\partial (c\psi)}{\partial t}+ \mathbf{D}(\x,\mathbf{u})\nabla c \cdot \nabla \psi\right)-\int_{t^{n}}^{t^{n+1}}\!\!\!\!\int_{\Omega} c(\phi \psi_{t} +\darcyU \cdot \nabla \psi)= \int_{t^{n}}^{t^{n+1}}\!\!\!\!\int_{\Omega} q_{c}\psi.
\end{equation}
A common feature of ELLAM schemes is taking test functions $\psi$ such that $ \phi \psi_{t}+\mathbf{u}^{n+1}\cdot \nabla \psi = 0$ in the
sense of distributions. This equation is completed by a terminal condition at
$t=t^{(n+1)}$, whose nature depends on the chosen approximation space for the concentration.
Here, the concentration is approximated by piecewise constant functions on the mesh,
and we therefore impose $\psi(t^{(n+1)},\cdot)=\mathbf{1}_K$, where $K$ is a generic cell
and $\mathbf{1}_{K}$ is the characteristic function of $K$. We will therefore denote $\psi=\psi_K$.

Computing these tests functions $\psi_K$ requires us to use the characteristics of the 
advective component of \eqref{conc.weak}, that is the solution, for each
$\mathbf{x}\in\O$, to the ODEs
\begin{equation}\label{charac}
\dfrac{d\widehat{\mathbf{x}}}{dt}(t) = \dfrac{\mathbf{u}^{(n+1)}(\widehat{\mathbf{x}}(t),t)}{\phi(\widehat{\mathbf{x}}(t))},  \qquad \widehat{\mathbf{x}}(t^{(n+1)})=\mathbf{x}.
\end{equation}
 This leads to
$\psi_K(t^{(n)},\mathbf{x})=\mathbf{1}_K(\widehat{\mathbf{x}})=\mathbf{1}_{\widehat{K}}(\mathbf{x})$,
where $\widehat{K}$ is $K$ tracked back from $t^{(n+1)}$ to  $t^{(n)}$
through \eqref{charac}:
\[
\widehat{K}=\{\x(t^{(n)})\,:\, \x(t^{(n+1)})\in K\}
\]
Of course, solving this characteristic equation first requires access to a velocity $\mathbf{u}^{n+1}$.
We describe in Section \ref{KRDarcyU} how to reconstruct, from the discrete Darcy fluxes obtained
in Section \ref{numPress}, such a velocity.

The diffusive term is discretised using the HMM method and an implicit time stepping
scheme. Fluxes $D_{K,\edge}^{(n+1)}$ are defined as for the pressure equation \eqref{pressure}, using a
piecewise constant Darcy velocity $\mathbf{U}^{(n+1)}$ constructed from the consistent part of the reconstructed pressure gradient \eqref{def.grad} and the viscosity at $c^{(n)}$, that is
\[
\forall K\in\mesh\,,\;\mathbf{U}^{(n+1)}_{|K}=\dfrac{\mathbf{K}_K}{\mu(c^{(n)}_K)}\ograd_K p^{(n+1)}.
\]
Hence, if $c^{(n+1)}\in X_{\mathcal{T}}$ is our discrete concentration sought at time $t^{(n+1)}$,
the fluxes are given by 
\begin{multline*}
\forall K\in\mesh\,,\;\forall v \in X_{\mathcal{T}}\,,\\
\sum_{\edge \in \edgescv} D_{K,\sigma}^{(n+1)} (v_{K}-v_{\sigma}) = \int_{K} \mathbf{D}(\x,\mathbf{U}^{(n+1)}_K)\nabla_H c^{(n+1)}(\x)
\cdot \nabla_H v(\x) d\x.
\end{multline*}
Note that, on the cell $K$, $\mathbf{D}(\cdot,\mathbf{U}^{(n+1)}_K)$ is actually constant.

Without discretising the source term yet (this will be discussed in Section \ref{numSource}), this leads to the following scheme for \eqref{conc.weak}:
\begin{equation}\nonumber
\forall K\in\mesh\,,\; \int_{K} \phi c^{(n+1)} -  \int_{\widehat{K}} \phi c^{(n)} +\Delta t\sum_{\sigma\in\mathcal{E}_{K}} D_{K,\sigma}^{(n+1)}= \int_{t^{n}}^{t^{n+1}}\hspace*{-0.8em}\int_\O q_{c^{(n+1)}}\psi_K.
\end{equation}

The integrals are then computed by writing
\[
\int_{K} \phi c^{(n+1)} = \phi_K |K| c^{(n+1)}_{K}\mbox{ and }\int_{\widehat{K}} \phi c^{(n)}=\sum_{M\in\mathcal{M}} \phi_M |\widehat{K}\cap M| c_{M}^{(n)},
\]
which leads to the following discretised form of the concentration equation
\begin{multline}\label{scheme:c}
\forall K\in\mesh\,,\;\\
\phi_K |K| c_{K}^{(n+1)} +\Delta t\sum_{\sigma\in\mathcal{E}_{K}} D_{K,\sigma}^{(n+1)}
=\sum_{M\in\mathcal{M}} \phi_M |\widehat{K}\cap M| c_{M}^{(n)} +\int_{t^{n}}^{t^{n+1}}\hspace*{-0.8em}\int_\O q_{c^{(n+1)}}\psi_K.
\end{multline}

As with the pressure equation, we also have the conservativity of fluxes and no flow boundary conditions given by 
\begin{equation}\label{scheme:c.BC}
\begin{aligned}
D_{K,\sigma}^{(n+1)}+D_{L,\sigma}^{(n+1)}=0 {}&\quad \mbox{ for all edge $\edge$ between two cells $K\not= L$,}\\
D_{K,\sigma}^{(n+1)}=0{}&\quad\mbox{ for all edge $\edge$ of $K$ lying on $\partial\O$}.
\end{aligned}
\end{equation}

\begin{remark}[Order of accuracy of the scheme]
	Under regularity assumptions on the mesh and the solutions, error estimates in $H^1$ norm for an HMM scheme for an anisotropic diffusion equation similar to \eqref{pressure} are of order $O(\Delta x)$ \cite[Theorem 3.31 and Propositions 13.15, 13.16]{DEGGH16}. Error estimates for an $\mathbb{RT}_0$--ELLAM scheme for an advection--diffusion equation similar to \eqref{concentration} are obtained in \cite{W05-ELLAM-error-estimates}, and are of order $O(\Delta x +\Delta t)$. However, for the complete Peaceman model \eqref{eq:model}, no exact error estimate may be established due to the fact that most of the regularity assumptions violate the physical interpretation of the problem. Nevertheless, convergence of the scheme can still be established, see \cite{CDL17-convergence-ELLAM}.
\end{remark}

\begin{remark}[Splitting technique]
The HMM--ELLAM for the concentration equation may also be viewed as an operator splitting technique, which is the usual approach when implementing ELLAM schemes \cite{AH06, RC-02-overview}. 
Introduce the values $\tilde{c}^{(n+1)}=(\tilde{c}^{(n+1)})_{K\in\mesh}$ defined by
\begin{equation}\label{scheme:c.ELLAM}
\forall K\in\mesh\,,\quad
\phi_K |K| \tilde{c}_{K}^{(n+1)} 
=\sum_{M\in\mathcal{M}} \phi_M |\widehat{K}\cap M| c_{M}^{(n)} +\int_{t^{n}}^{t^{n+1}}\hspace*{-0.8em}\int_\O q_{c^{(n+1)}}\psi_K.
\end{equation}
The approximate values $\tilde{c}^{(n+1)}$ correspond to applying one time step of the ELLAM
technique on the pure advection equation $\phi \partial_t c + \nabla \cdot (\mathbf{u}c) = q_{c}$. Re-starting from $\tilde{c}^{(n+1)}$, construct now $c^{(n+1)}\in X_{\mathcal T}$
such that
\begin{equation}\label{scheme:c.DIFF}
\forall K\in\mesh\,,\quad
\phi_K |K| (c_{K}^{(n+1)}-\tilde{c}^{(n+1)}_K) +\Delta t\sum_{\sigma\in\mathcal{E}_{K}} D_{K,\sigma}^{(n+1)}
=0
\end{equation}
with boundary conditions \eqref{scheme:c.BC}. In other words, $c^{(n+1)}$ is obtained from
$\tilde{c}^{(n+1)}$ by performing one time step of the HMM scheme for the pure
diffusion equation $\phi \partial_t c -\grad\cdot(\mathbf{D}(\mathbf{x},\mathbf{u})\nabla c) = 0$. Adding together \eqref{scheme:c.ELLAM} and \eqref{scheme:c.DIFF}, we see
that $c^{(n+1)}$ is the solution to the complete HMM--ELLAM scheme \eqref{scheme:c}.
\end{remark}

\subsection{Reconstruction of Darcy velocities} \label{KRDarcyU}

In order to compute the characteristics in \eqref{charac}, two important features of the velocity need to be accounted for: the no-flow
boundary conditions $\mathbf{u}^{(n+1)}\cdot\mathbf{n}=0$ on $\partial\O$, which
ensures that the solutions to \eqref{charac} do not exit the computational domain,
and the preservation of the divergence in \eqref{pressure}, to avoid creating
regions with artificial wells or sinks (which lead to non-physical flows), as was mentioned and implemented in \cite{AH10-Fully-Conservative-ELLAM,AW11-stability-monotonicity-implementation}.
We present this construction in dimension $d=2$, the extension to dimension $d=3$ being the purpose of a
future work.

Preserving these features is done by using a technique similar to that of \cite{KR03-kuznetsov-repin}. Each cell
$K\in\mesh$ is split into
triangles (see Fig. \ref{fig.subd}) by choosing an interior point $\x_{K}$ and connecting it to the vertices of cell $K$.
These vertices are denoted by $(\v_1,\ldots,\v_r)$, in counter clockwise order, and with the convention that $\v_0=\v_r$.
We let $\edge_i^*$ be the internal edge $[\v_i,\x_K]$ of the triangular subdivision of $K$, and $\edge_i$ be the
edge $[\v_{i-1},\v_i]$ of $K$.
  A clockwise oriented interior flux $F_{\sigma_i^*}$ is then computed on each internal edge $\edge_i^*$ of this subdivision. $\mathbf{u}^{(n+1)}$ is the $\RT0$ function reconstructed from these
fluxes on the triangular subdivision. 

This reconstruction belongs to $H_\div(\O)$ and,
to ensure that its divergence is identical to the discrete
divergence computed from the Darcy fluxes, the
internal fluxes $F_{\sigma_i^*}$ are constructed so that their balance (along
with the fluxes $F_{K,\sigma_i}^{(n+1)}$ at the boundary of $K$) in each
triangle corresponds to the balance over the cell $K$:
\begin{equation} \label{locDiv2D}
\dfrac{1}{|T_{K,\edge_i}|}\left(F_{\edge_{i}^*}-F_{\edge_{i-1}^*} + F_{K,\edge_i}^{(n+1)}\right) = \dfrac{1}{|K|} \sum_{\edge\in\edgescv}F_{K,\sigma}^{(n+1)}, \quad \text{ for } i=1,2,\dots,r.
\end{equation}

\begin{figure}[h]
	\caption{Triangulation of a generic cell.}
	\begin{center}
		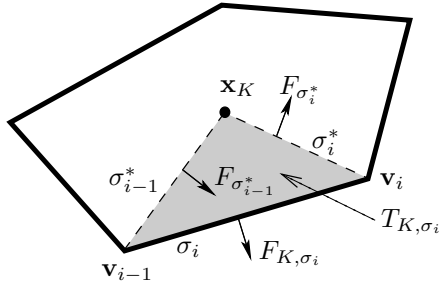
	\end{center}
	\label{fig.subd}
\end{figure}

\subsubsection{KR velocity}\label{sec:KRvel}

Although \eqref{locDiv2D} gives us $r$ equations in $r$ unknowns, this system of equations is underdetermined. More specifically, its rank is $r-1$ and it therefore has an infinite number of solutions. The original method
of Kuznetzov and Repin \cite{KR03-kuznetsov-repin} consists in selecting the solution that has a minimal $l^2$ norm. The velocities reconstructed from these fluxes will be referred to as \emph{KR velocities}. 

Despite its applicability to any space dimension, this choice is questionable from the physical point of view. If, say, the
external fluxes $(F_{K,\edge}^{(n+1)})_{\edge\in\edgescv}$ correspond to those of a constant velocity in the cell, this reconstruction does not necessarily recover this natural velocity.

\subsubsection{C velocity}\label{sec:Cvel}

We therefore propose a different approach, which consists in adding to the set of equations \eqref{locDiv2D} a
complementary equation which selects a unique solution. We form this complementary equation by ensuring that,
as \eqref{locDiv2D}, it is consistent for internal and external fluxes coming from simple
velocities -- here, not just constant velocities in the cell, but $\RT0$ velocities in the cell.

\begin{lemma}[Consistency condition] \label{ConsistencyG}
	Let $\x_{K}\in K$, $(\v_{1},\v_{2},\dots, \v_{m})$ be some vertices of cell $K$, and $(\alpha_1,\ldots,\alpha_m)\in\R^m$
be such that
	\begin{equation}\nonumber
	\sum_{i=1}^{m} \alpha_{i} \v_{i} = \x_{K} \quad \text{and} \quad \sum_{i=1}^{m}\alpha_{i}=1.
	\end{equation}
Let $\darcyU$ be an $\RT0$ function over the cell $K$, that is, $\darcyU(\x)=a\x+\mathbf{b}$ for
some $a\in\R$ and $\mathbf{b}\in\R^d$. Define
$
F_{\edge_{i}^{*}} =\int_{\edge_{i}^{*}} \darcyU \cdot \bfn_{\edge_{i}^{*}}
$
the flux of $\darcyU$ through the internal edge $\edge_i^*$, for all $i=1,\ldots,m$ (here,
$\bfn_{\edge_{i}^{*}}$ is the clockwise normal to $\edge_i^*$,
as in Figure \ref{fig.subd}).
Then 
	\begin{equation}\nonumber
	\sum_{i=1}^{m} \alpha_{i}F_{\edge_{i}^{*}}=0.
	\end{equation} 
\end{lemma}
\begin{proof}
	Denoting by $x_{\edge_{i}^{*}}$ the center of $\edge_{i}^{*}$ and $\mathrm{Rot}$ the clockwise rotation by $\pi/2$, write
	\[
	F_{\edge_{i}^{*}} = \int_{\edge_{i}^{*}} \darcyU \cdot \bfn_{\edge_{i}^{*}} = (a \x_{\edge_{i}^{*}} + \mathbf{b}) \cdot \bfn_{\edge_{i}^{*}} |\edge_{i}^{*}| 
	= (a \x_{\edge_{i}^{*}} + \mathbf{b}) \cdot \mathrm{Rot}(\x_{K}-\v_{i}).
	\]
	Since $(\x_{\edge_i^*}-\x_K)$ is orthogonal to $\mathrm{Rot}(\x_{K}-\v_{i})$, we deduce
	\[
	\begin{aligned}
	\sum_{i=1}^{m} \alpha_{i}F_{\edge_{i}^{*}}&=  \sum_{i=1}^{m} \alpha_{i}(a \x_{\edge_{i}^{*}} + \mathbf{b}) \cdot \mathrm{Rot}(\x_{K}-\v_{i})  \\
	&=  \sum_{i=1}^{m} \alpha_{i}(a (\x_{\edge_{i}^{*}}-\x_{K})+a\x_{K} + \mathbf{b}) \cdot \mathrm{Rot}(\x_{K}-\v_{i}) \\
	&= (a\x_{K} + \mathbf{b}) \cdot\mathrm{Rot}\left(\sum_{i=1}^{m} \alpha_{i} (\x_{K}-\v_{i})\right).
	\end{aligned}
	\]
The proof is complete by noticing that $\sum_{i=1}^{m} \alpha_{i} (\x_{K}-\v_{i})=0$
by choice of the coefficients $(\alpha_i)_{i=1,\ldots,m}$. \end{proof}

This lemma tells us that by choosing 3 or more vertices of our cell $K$ in generic position, we can find an equation that closes the system \eqref{locDiv2D}. We only need one such equation, and we do not want to create a bias
in constructing it. We therefore use all vertices $(\v_1,\ldots,\v_r$) of $K$ to form this relation.
 Thus our closing equation will be given by 
\begin{equation} \label{closingEqn}
\sum_{i=1}^{r} \alpha_{i}F_{\edge_{i}^{*}}=0,
\end{equation} 
where the coefficients $(\alpha_i)_{i=1,\ldots,r}$ are related to $\x_K$ as in Lemma \ref{ConsistencyG}. Velocities reconstructed from fluxes that satisfy \eqref{locDiv2D}--\eqref{closingEqn} will be denoted as \emph{C velocities}
(`C' for `consistent').

\begin{remark}[Particular $\x_K$ and barycentric combinations]
If $\x_{K}$ is the iso-barycenter of the vertices of $K$, i.e. $\x_{K}= \frac{1}{r} \sum_{i=1}^{r} v_{i}$, then the
consistency relation is simply $\sum_{i=1}^{r} F_{\edge_{i}^{*}}=0$.
If $\x_{K}$  is the center of mass of $K$, then a consistency relation is
\[
\sum_{i=1}^{r} \dfrac{|T_{K,\edge_{i-1}}|+|T_{K,\edge_{i}}|}{2|K|} F_{\edge_{i}^{*}}=0,
\]
where $T_{K,\edge_{i-1}}$ is the triangle that shares edge $\edge_{i-1}^{*}$ with $T_{K,\edge_i}$.
\end{remark}

\medskip

The system \eqref{locDiv2D}--\eqref{closingEqn} has a unique and explicit solution.
Indeed, set 
\[
a_{i}=\dfrac{|T_{K,\edge_{i}}|}{|K|}\left(\sum_{\edge\in\edges_{K}} F_{K,\edge}^{(n+1)}\right) -F_{K,\edge_{i}}^{(n+1)}\qquad\mbox{ for $i=1,\ldots,r$}.
\]
The system \eqref{locDiv2D} is then
\[
\begin{aligned}
F_{\edge^*_{1}} &= F_{\edge^*_{r}}+a_{1} \\
F_{\edge^*_{2}} &= F_{\edge^*_{1}}+a_{2} \\
&\vdots \\
F_{\edge^*_{r-1}} &= F_{\edge^*_{r-2}}+a_{r-1} \\
\end{aligned}
\]
From these, we easily deduce that
\begin{equation} \label{explicitK}
F_{\edge^*_{k}}=F_{\edge^*_{r}}+\sum_{j=1}^{k}a_{j},\quad  k=1,2,\dots,r-1
\end{equation}
By noticing that $\sum_{i=1}^r a_i=0$, we see that \eqref{explicitK} also holds for $k=r$.
Together with the closing relation \eqref{closingEqn}, and since $\sum_{k=1}^r \alpha_k=1$, we obtain an explicit expression for $F_{\edge^*_{r}}$, given by
\begin{equation}\label{explicitN}
F_{\edge^*_{r}}=-\sum_{k=1}^{r} \left(\alpha_{k}\sum_{j=1}^{k}a_{j} \right).
\end{equation}
Equation \eqref{explicitN}, together with \eqref{explicitK}, will then give us explicit expressions for $F_{\edge^*_{k}},$  $k=1,2,\dots, r$. 

\medskip

These computations show an advantage of this method over the technique consisting in selecting
a minimal norm solution of \eqref{locDiv2D}. Here, we do not need to solve any local system, as we have
explicit expressions of the solution to \eqref{locDiv2D}--\eqref{closingEqn}. Moreover, 
the reconstructed velocity is more precise, especially on skewed meshes.
As an example, consider a  constant velocity field $V=(0,1)$. Pick a cell $K$
and define the fluxes across its boundary edges to be those of $V$. Denote then by $V_{\rm KR}$ and $V_{\rm C}$ the KR and C velocities, respectively, reconstructed as above on a triangular sub-mesh of $K$.
 
Table \ref{tab:recons} represents the relative error obtained between the exact and reconstructed velocities,
on a variety of mesh geometries.
For square cells (Cartesian mesh), both methods reconstruct the velocity accurately. However, for cells from
hexahedral (Fig. \ref{CHmeshes}, right), non-conforming, and Kershaw meshes (Fig. \ref{NKmeshes}), KR velocities noticeably deviate from the actual velocity,
by more than $30\%$ on distorted cells. On the other hand, as expected, using the consistency relation as a closure equation enables us to recover the velocity $V$ up to machine precision, regardless of the mesh. 

\begin{table}[h]
	\begin{center}
		\begin{tabular}{|c|c|c|}
			\hline
			Mesh & $\dfrac{||V-V_{\rm KR} ||}{||V||}$ & $\dfrac{||V-V_{\rm C} ||}{||V||}$ \\
			\hline
			Cartesian & 2.9038e-15  & 5.9529e-16 \\
			\hline
			Hexahedral  & 0.0371 & 5.8098e-15 \\
			\hline
			Non-conforming & 0.0348 & 6.7103e-15 \\
			\hline
			Kershaw  & 0.3151 & 4.6738e-15 \\
			\hline
		\end{tabular} 
	\end{center}
	\caption{Relative errors in velocity reconstruction.}\label{tab:recons}
\end{table}

\subsection{Approximate traceback region, and tracking points through vertices} \label{approx_traceback}

In general, we cannot get an exact representation of the region $\widehat{K}$, and thus, for each cell $K$, the traceback region $\widehat{K}$ is approximated in the following manner: we select points $(\x_i)_{i=1,\ldots,\ell_K}$
along the boundary of $K$ (at least all the vertices and edge midpoints are selected), we solve
\eqref{charac} starting from any of these points, thus getting curves $(\widehat{\x}_i)_{i=1,\ldots,\ell_K}$,
and we approximate $\widehat{K}$ by the polygon defined by the points $(\widehat{\mathbf{x}}(t^{(n)}))_{i=1,\ldots,\ell_K}$.

The reconstructed velocity $\mathbf{u}^{(n+1)}$ is an $\RT0$ function on a triangular sub-mesh.
Tracking a point through \eqref{charac} is naturally done cell-by-cell, using the
value of $\mathbf{u}^{(n+1)}$ in a cell $K$ as long as $\widehat{\x}$ stays in $K$, and then, when $\widehat{\x}$
exits $K$ to enter a cell $L$, continuing the tracking by using the value of $\mathbf{u}^{(n+1)}$ in $L$. This type of tracking, with point of exit determined by minimal time of flight, was first implemented by Pollock \cite{P88-Pollock-Tracing} on meshes characterised by orthogonal grid blocks, e.g. Cartesian meshes. Pollock's algorithm was then extended to more generic types of cells in \cite{PEB02-generalised-Pollock}, and further improved in \cite{JSDK07-improved-Pollock}. 
Because the fluxes of $\mathbf{u}^{(n+1)}$ are continuous across the edges, this tracking procedure ensures
that a point will never do a U-turn, that is, $-\mathbf{u}^{(n+1)}_{|L}$ (we use $-\mathbf{u}^{(n+1)}$ since we are tracking backwards) will not force $\widehat{\x}$ to re-enter
$K$ before even entering $L$ (this would in effect freeze $\widehat{\x}$ on the interface between $K$ and $L$). 

\medskip

During this tracking, special care must be taken with points that start or pass through a vertex.
An initial position $\x$ corresponding to a vertex is involved in several triangles, and could thus be 
initially tracked using any of the Darcy velocities in these triangles (see Fig. \ref{vertices-init}). 
To avoid non-physical initial tracking, we compute the Darcy velocities at the vertex $\x$ in each of the triangles involved with it. Picking one of these Darcy velocities at random is not acceptable, since if its opposite vector points outside the corresponding triangle, this means that $\x$ would never be tracked back inside this triangle, and that the chosen Darcy velocity is thus not the correct one.

\begin{figure}[h]
	\centering
	\includegraphics{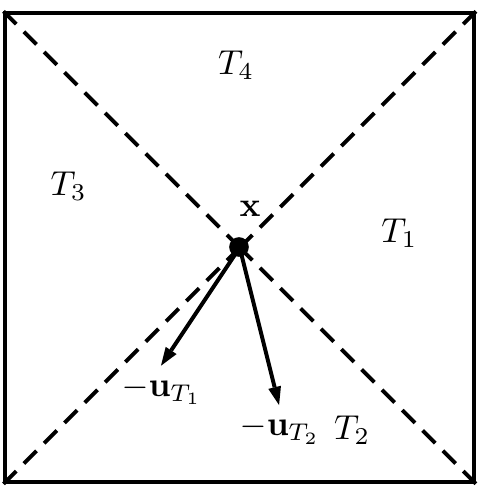}
	\caption{Choosing the proper triangle to initialize the tracking.}
	\label{vertices-init}
\end{figure}

\begin{figure}[h]
	\centering
\includegraphics{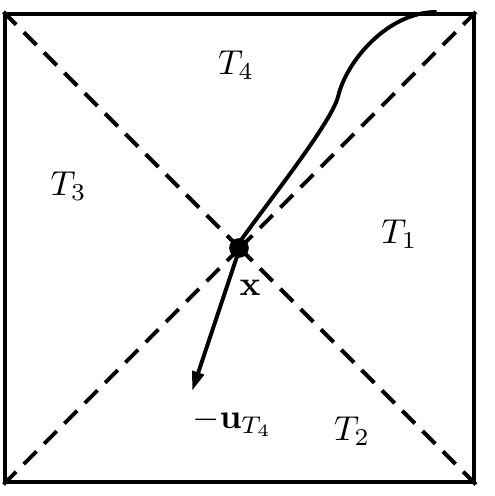}
	\caption{Choosing the proper triangle to continue the tracking. Here $\darcyU_{T_{4}}$ represents the vector that we obtain by computing the Darcy velocity on point $\x$ using the reconstructed velocity at $T_{4}$.}
	\label{tracking-vertices}
\end{figure}

We therefore loop over the triangles $T_1$, $T_2$, etc. until we reach a triangle $T_{n}$ such that $-\mathbf{u}^{(n+1)}_{|T_n}(\x)$ points inside $T_n$. In Figure 2, this corresponds to triangle $T_2$. As a note, regardless of the mesh, our numerical tests suggest that such a triangle always exists. For points that will be tracked into a vertex at some time in $[t^{(n)},t^{(n+1)})$ (see Fig. \ref{tracking-vertices}), we consider the negative of the Darcy velocity in the triangle it came from (in this case, $T_4$) and  determine which triangle it points into (in this case, it points into triangle $T_2$); the
tracking is then continued based on the reconstructed velocity in this latter triangle.

\subsection{Source term} \label{numSource}

The integral involving the source term should be treated carefully, otherwise the numerical results will feature severe overshoots, especially over the regions around the injection well. For our discretisation, we use a weighted trapezoid rule in time:
\begin{equation} \label{wtTrapRule}
\int_{t^{n}}^{t^{n+1}}\hspace*{-0.8em}\int_\O q_{c^{(n+1)}}\psi_K=w\Delta t\int_{\widehat{K}} q_{c^{(n+1)}}(t^{(n)})+(1-w)\Delta t\int_{K} q_{c^{(n+1)}}(t^{(n+1)})
\end{equation} 
where $q_{c^{(n+1)}}(s)=q^+(s)-q^-(s)c^{(n+1)}$ (that is, the source term is still fully implicit in $c$).

Note that the left and right quadrature rules correspond to $w=1$ and $w=0$, respectively. To determine the proper weight, we consider an injection cell $K=E$. We mainly focus on injection cells since this is where mass conservation might fail locally.
A proper weight that will yield mass conservation has been derived for Cartesian meshes on \cite{AH06}. This can easily be generalised for arbitrary meshes, and is given by
\begin{equation} \label{coeffWtTrapRule}
 w=\dfrac{1}{1-e^{-\alpha}}-\dfrac{1}{\alpha}, \quad \text{where} \quad \alpha= \dfrac{\int_{E} q(t^{(n+1)})}{\int_{E} \phi}\Delta t.
\end{equation}

Then, for each cell $K$ (injection or not), the integral of the source term is computed using the weighted trapezoid rule \eqref{wtTrapRule}, where $w$ is obtained as in \eqref{coeffWtTrapRule}, for some $E$ related to $K$ -- see below. We treat the computation of the integral over $\widehat{K}$ in the right hand side of \eqref{wtTrapRule} in different manners, depending on whether the cell $K$ is 
\begin{enumerate}
	\item[i)] an injection cell,
	\item[ii)] a cell tracked back into an injection cell (but not an injection cell itself),
	\item[iii)] or a cell that does not track back into an injection cell.
\end{enumerate}

i) If the cell $K$ is an injection cell $E$, then it tracks back entirely into itself. Hence, over
the entire interval $[t^{(n)},t^{(n+1)}]$, $\nabla\cdot \mathbf{u}^{(n+1)}(\widehat{x})=\frac{1}{|E|}\int_E q(t^{(n+1)})$ and
thus we use, through Liouville's theorem, an exact computation
\begin{equation}\label{inj_Liouville}
\int_{\widehat{K}} q_{c^{(n+1)}}(t^{(n)}) =  e^{-\alpha}\int_{K} q_{c^{(n+1)}}(t^{(n+1)})
\end{equation}
where $\alpha$ is given by \eqref{coeffWtTrapRule}.

ii) If the cell $K$ is not an injection cell, but is tracked back (at least partially) into an injection cell $E$, then we use a forward tracking algorithm similar to that described in \cite{AW11-stability-monotonicity-implementation}.  For these regions, we compute the \emph{trace forward} region $\widetilde{E}$ of $E$, in which we solve the characteristics \eqref{charac} with initial condition $\widehat{\mathbf{x}}(t^{(n)})=\x$ instead. The trace forward region $\widetilde{E}$ is then approximated using polygons in a similar manner as the trace back regions in Section \ref{approx_traceback}. The integral of the first term on the right hand of \eqref{wtTrapRule} is then approximated by
\begin{equation} \nonumber
\int_{\widehat{K}} q_{c^{(n+1)}} \approx \dfrac{|K\cap(\widetilde{E} \setminus E)|}{| \widetilde{E} \setminus E |} (1-e^{-\alpha}) \int_{E} q_{c^{(n+1)}}.
\end{equation} 
 In physical terms, this means that the volume injected from the well $E$ is transported into each of the cells $K$ proportionally. Note that, on the contrary to \cite{AW11-stability-monotonicity-implementation}, only a fraction $(1-e^{-\alpha})$ of $\int_{E} q_{c^{(n+1)}}$ is being spread in the cells $K$ around $E$, since a fraction $e^{-\alpha}$ of $\int_{E} q_{c^{(n+1)}}$ has already been allocated to $E$, as can be seen in \eqref{inj_Liouville}.

iii) Finally, if a cell $K$ does not track back into an injection cell, then the integral
$\int_{\widehat{K}} q_{c^{(n+1)}}$ will be computed using the (approximate) trace back regions
as described in Section \ref{numConc}. Actually, in that situation, either:
\begin{itemize}
\item $K$ is not a production cell, $\widehat{K}$ is disjoint from injection cells
and production cells (due to $-\mathbf{u}^{(n+1)}$ pointing outside production cells).
Therefore, the integrals in \eqref{wtTrapRule} are both equal to zero.
\item or $K$ is a production cell, in which case, by nature of the Darcy flow, it
is expected that $K \subset \widehat{K}$, so both integrals in \eqref{wtTrapRule}
are equal and the value of $w$ is irrelevant.
\end{itemize}

\section{Numerical results} \label{numResults}

We perform numerical simulations under the following standard test case (see, e.g., \cite{WLELQ-00}):
\begin{enumerate}
	\item $\O=(0,1000) \times (0,1000) \text{ ft}^{2}$, 
	\item injection well at $(1000,1000)$ and production well at $(0,0)$, both with flow rate of $30 \text{ft}^{2}/\text{day}$,
	\item  constant porosity $\phi=0.1$ and constant permeability tensor $\mathbf{K}=80\mathbf{I}$ md,
	\item oil viscosity $\mu(0)=1.0$ cp and mobility ratio $M=41$,
	\item $\phi d_{m}=0.0 \text{ft}^{2}/\text{day}$, $\phi d_{l}=5.0\text{ft}$, and $\phi d_{t}=0.5 \text{ft}$
\end{enumerate}

For the time discretisation, we take a time step of $\Delta t =36$ days. These will be simulated on Cartesian type meshes (square cells of dimension $62.5 \times 62.5$ ft), hexahedral meshes (see Fig.\ref{CHmeshes}), non-conforming meshes, and finally on Kershaw type meshes, as described in \cite{HH08}  (see Fig. \ref{NKmeshes}). 

\begin{figure}[h]
	\centering
	\begin{tabular}{c@{\hspace*{2em}}c}
		\includegraphics[width=0.4\textwidth]{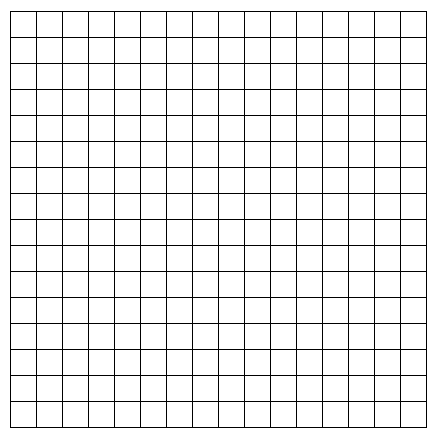} & 		\includegraphics[width=0.4\textwidth]{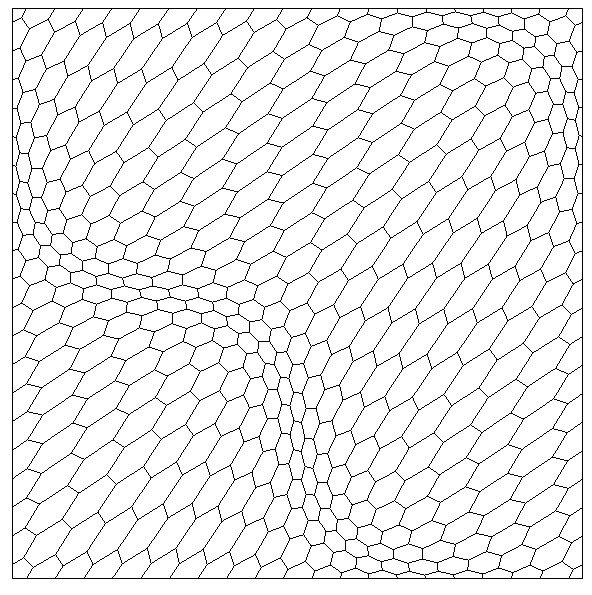}\\
	\end{tabular}
	\caption{ Mesh types(left: Cartesian ; right: hexahedral).}
	\label{CHmeshes}
\end{figure}
\begin{figure}[h]
	\centering
	\begin{tabular}{c@{\hspace*{2em}}c}
		\includegraphics[width=0.4\textwidth]{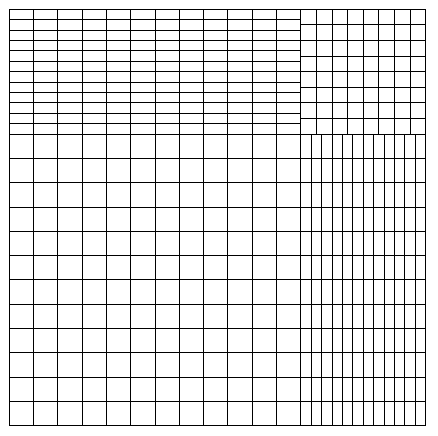} & 		\includegraphics[width=0.4\textwidth]{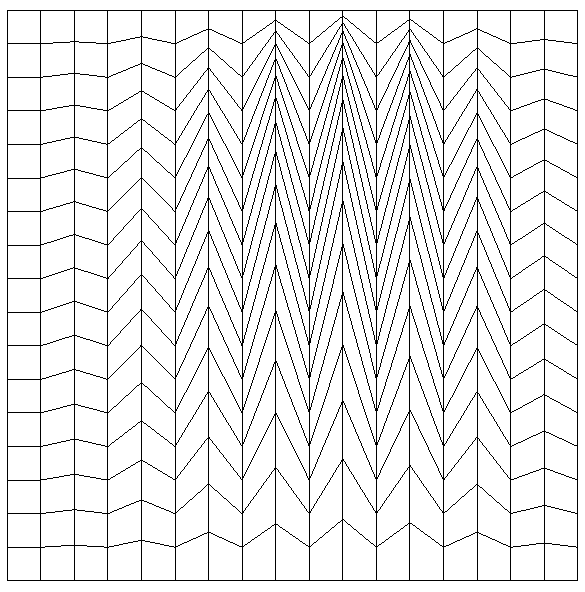}\\
	\end{tabular}
	\caption{ Mesh types(left: non-conforming ; right: Kershaw).}
	\label{NKmeshes}
\end{figure}

\subsection{Effect of the quadrature rule}

The following simulations are based on KR velocities
(see Section \ref{sec:KRvel}). Figure \ref{CLR} shows the numerical solution for the concentration at $t=10$ years on a Cartesian mesh using the left and the right rule, respectively. These results support the observations made in \cite{CD17-HMM-ELLAM}, i.e., that the left and right hand quadrature rules provide severe underestimates and overshoots, respectively, at the injection well, and are thus not good choices.

\begin{figure}[h]
	\centering
	\begin{tabular}{c@{\hspace*{2em}}c}
		\includegraphics[width=0.45\textwidth]{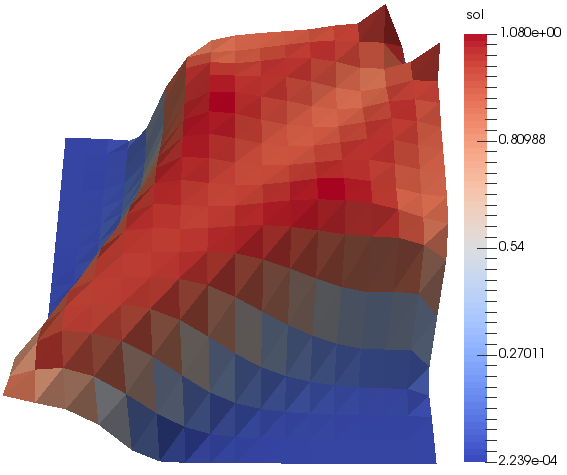} & 		\includegraphics[width=0.45\textwidth]{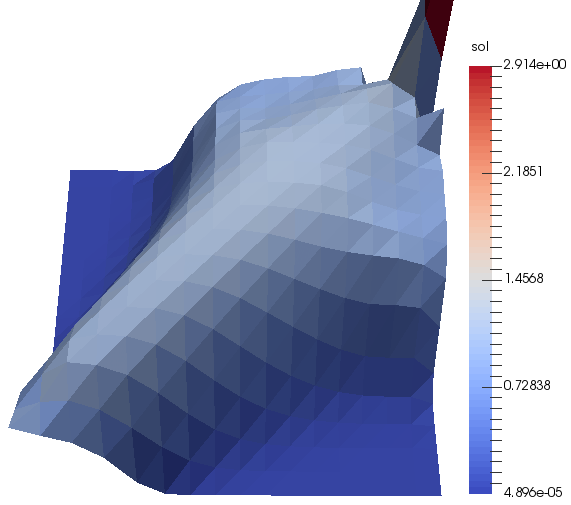}\\
	\end{tabular}
		\caption{Cartesian mesh, $t=10$ years, KR velocities (left: left rule for source terms; right: right rule for source terms).}
		\label{CLR}
\end{figure}


Figure \ref{AC310_l2} (left) shows the numerical solution for the concentration using the proper weight for the trapezoidal rule, and computation of the integrals as described in Section \ref{numSource}. This presents a significant improvement from the results obtained through the right and left rule. 

The overshoot seen in this figure is at worst around $7\%$, which is
commensurate with (or even less than) overshoots already noticed in other
other characteristic methods in the absence of specific post-processing \cite{ERW83}.

\subsection{Comparison with forward tracking in \cite{AW11-stability-monotonicity-implementation}}

We now compare our numerical results to the original algorithm of \cite{AW11-stability-monotonicity-implementation}.
Instead of implementing i) and ii) as described in Section \ref{numSource},
\begin{itemize}
\item For an injection cell $E$, $c_E^{(n+1)}$ is fixed at $1$, as this is the concentration of the injected solvent.
\item for cells $K$ tracked back (at least partially) into an injection cell $E$, the following approximation
is used:
	\begin{equation} \label{A_dist}
	\int_{\widehat{K}} q_{c^{(n+1)}} \approx \dfrac{|K\cap(\widetilde{E} \setminus E)|}{| \widetilde{E} \setminus E |} \int_{E} q_{c^{(n+1)}},
	\end{equation}
where $\widetilde{E}$ is the trace forward region of $E$.
\end{itemize}

For some discretisation parameters, this implementation might lead to degraded results due to its physical implications. Setting  $c_E^{(n+1)}=1$ corresponds to distributing a fraction of $\int_{E} q_{c^{(n+1)}}$ into injection cells $E$. In this instance, a good estimate would be given by \eqref{inj_Liouville}. However, computation of the integral for cells $K$ that track back into an injection cell $E$ by \eqref{A_dist} means that we spread the whole of $\int_{E} q_{c^{(n+1)}}$ onto the cells $K$. This then means that at time level $n+1$, an excessive amount of $e^{-\alpha}\int_{E} q_{c^{(n+1)}}$ of fluid has been injected in the cells around $E$. If $\Delta t$ is large enough, then this is a negligible excess as $e^{-\alpha} \approx 0$. However, for moderate to small $\Delta t$, the numerical results do not model the physical phenomenon properly. It is important to note that even though characteristic methods aim for computations using large time steps, we should still have a proper numerical result even when the time steps are small.

\begin{figure}[h]
	\centering
	\begin{tabular}{c@{\hspace*{2em}}c}
		\includegraphics[width=0.45\textwidth]{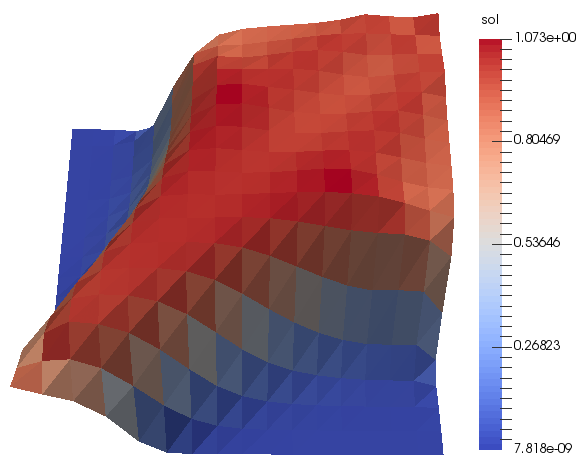} & 		\includegraphics[width=0.45\textwidth]{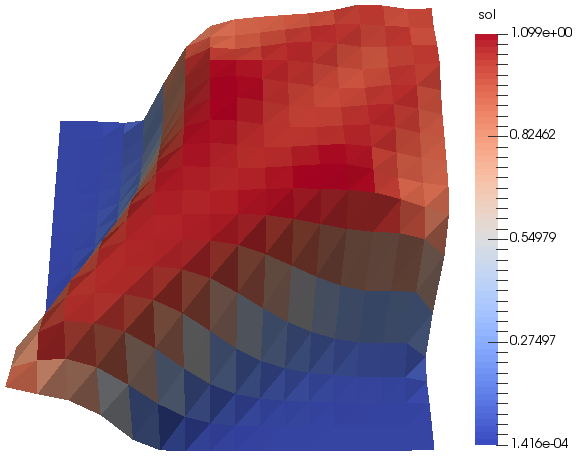}
	\end{tabular}
	\caption{Cartesian mesh, weighted trapezoid rule for source terms, KR velocities, $t=10$ years  (left: trace forward as in Section \ref{numSource}; trace forward as in \cite{AW11-stability-monotonicity-implementation}).}
	\label{AC310_l2}
\end{figure}

Figure \ref{AC310_l2} (right) shows the numerical solutions obtained at $t=10$ years upon computing the integrals as in \cite{AW11-stability-monotonicity-implementation}, with the moderate time step $\Delta t=36$ days. As expected, due to injection of too much fluid, the overshoot at the right of Figure \ref{AC310_l2} (around 9.9\%) is larger than the one on the left (around 7.3\%). This particular feature is even more evident if we take smaller time steps. Due to this, we see that the implementation we propose in Section \ref{numSource} is more accurate.

\subsection{A criterion for choosing the number of points per edge}
In general, a polygon formed by tracking back only vertices and edge midpoints might not give a good approximation to the trace back region $\hat{K}$. The mesh regularity parameter, defined as 
	\begin{equation}\nonumber
m_{\mesh \text{reg}} :=\max_{K\in\mesh}\dfrac{\mathrm{diam}(K)^{2}}{|K|},
\end{equation}
has been used as a criterion for determining the proper number of points to track per edge (see Table \ref{tab:reg}). To be exact, $\lceil \log_{2}(m_{\mesh \text{reg}})\rceil$ points per edge should be tracked in order to obtain a reasonable numerical solution. These results have been established in \cite{CD17-HMM-ELLAM} for KR velocities. For this paper, we verify that this still holds for C velocities, and on the previously not considered non-conforming mesh.
\begin{table}[h]
	\begin{center}
		\begin{tabular}{|c|c|c|c|}
			\hline
			Mesh & $m_{\mesh\text{reg}}$ & $\log_{2}(m_{\mesh\text{reg}})$ & points per edge \\
			\hline
			Cartesian & 2  & 1 & 1 \\
			\hline
			Hexahedral & 5.4772 & 2.4534 & 3 \\
			\hline
			Non-conforming & 2.7619 & 1.4657 & 2 \\
			\hline
			Kershaw & 32.0274 & 5.0012 & 6 \\
			\hline
			
		\end{tabular} 
	\end{center}
	\caption{Regularity parameter of the meshes and nb of points to approximate
		the trace-back cells.}
	\label{tab:reg}
\end{table} 

\begin{figure}[h]
	\centering
	\begin{tabular}{c@{\hspace*{2em}}c}
		\includegraphics[width=0.45\textwidth]{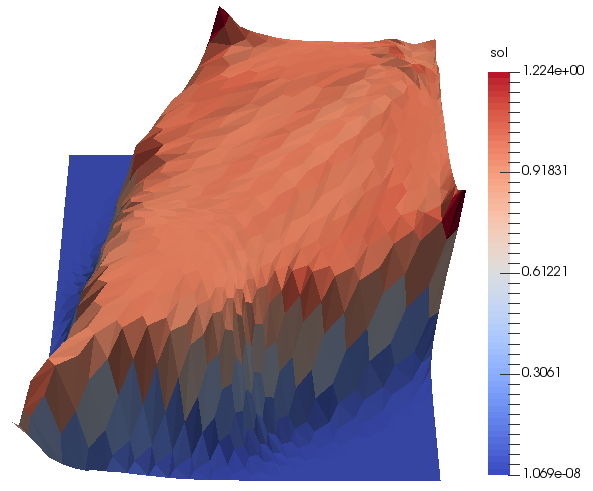} & 		\includegraphics[width=0.45\textwidth]{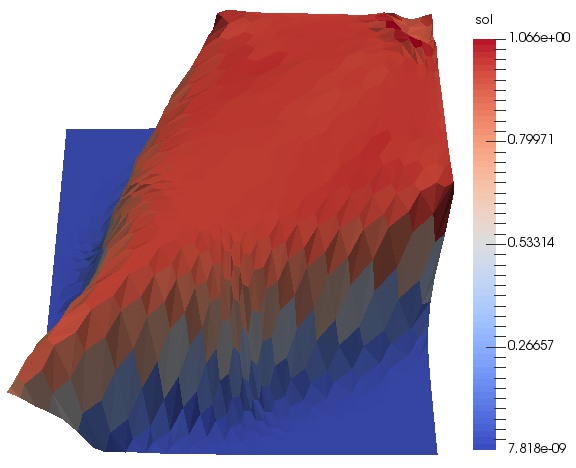}
	\end{tabular}
	\caption{Hexahedral mesh, weighted trapezoid rule for source terms, C velocities, $t=10$ years (left: edge midpoint, right: 3 points per edge).}
	\label{H310_l2}
\end{figure}

We start by demonstrating on hexahedral meshes that even for C velocities, tracking only vertices and edge midpoints do not give a good approximation (see Fig. \ref{H310_l2} left). As expected, taking 3 points per edge, as suggested in Table \ref{tab:reg}, then gives a better result, with an overshoot $< 7\%$ (see Fig. \ref{H310_l2} right). 

This heuristic choice of number of points along each edge is further backed up by the numerical solutions for the non-conforming meshes, and also for the very distorted `Kershaw' meshes, see Figures \ref{nCon_l2} and \ref{CK310}. 

We note that not all cells are highly "irregular"; thus, tracking a lot of points along each edge for the whole mesh introduces unnecessary numerical cost. If the cell $K$ is an injection or production cell, then we still track, as before, $\lceil \log_{2}(m_{\mesh \text{reg}})\rceil$ points along each edge. As an improvement, if $K$ is neither an injection or production cell, we determine the number of points to track along each edge of cell $K$ by measuring instead the cell regularity parameter defined to be 
\be \nonumber
m_{K\text{reg}} := \dfrac{\text{diam}(K)^2}{|K|}
\ee  and track $\lceil \log_{2}(m_{K \text{reg}})\rceil$ points along each edge of cell $K$.  By doing so, we reduce the computational cost without degrading the quality of the numerical solutions. 
\subsection{Comparison of the two reconstructions of the Darcy velocity}
In this section, we compare the numerical solutions at $t=10$ years obtained through simulations using KR velocities to those obtained when we do our simulations using C velocities. These will be performed over Cartesian, hexahedral, non-conforming and Kershaw type meshes using the recommended number of points per edge (see Table \ref{tab:reg}) for tracking.
\begin{figure}[h]
	\begin{tabular}{cc}
		\includegraphics[width=0.4\linewidth]{Cart_10} & 		\includegraphics[width=0.4\linewidth]{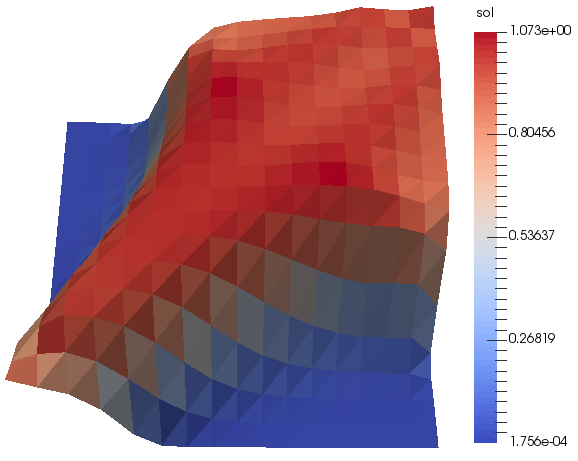}
	\end{tabular}
	\caption{Cartesian mesh, weighted trapezoid rule for source terms, $t=10$ years (left: KR velocities, right: C velocities).}
	\label{CC310}
\end{figure}

\begin{figure}[h]
	\centering
	\begin{tabular}{c@{\hspace*{2em}}c}
		\includegraphics[width=0.45\textwidth]{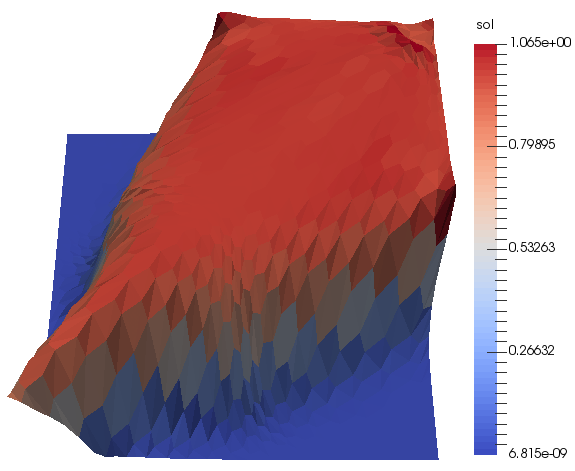} & 		\includegraphics[width=0.45\textwidth]{Hexa_3_10_cons}
	\end{tabular}
	\caption{Hexahedral mesh, weighted trapezoid rule for source terms, $t=10$ years (left: KR velocities, right: C velocities).}
	\label{H310}
\end{figure}

\begin{figure}[h]
	\centering
	\begin{tabular}{c@{\hspace*{2em}}c}
		\includegraphics[width=0.45\textwidth]{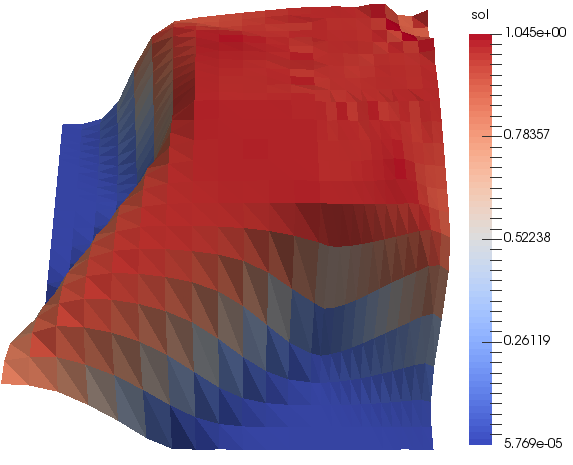} & 		\includegraphics[width=0.45\textwidth]{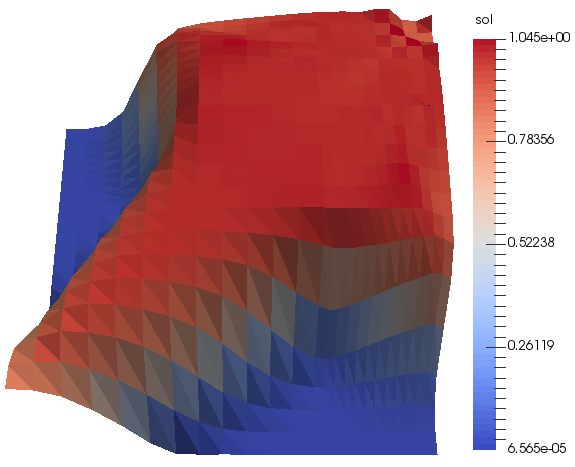}
	\end{tabular}
	\caption{Non-conforming mesh, 2 points per edge, weighted trapezoid rule for source terms, $t=10$ years (left: KR velocities, right: C velocities).}
	\label{nCon_l2}
\end{figure}

Figures \ref{CC310} to \ref{nCon_l2} compare the numerical solutions at $t=10$ years, on Cartesian, hexahedral, and non-conforming meshes, between KR velocities and C velocities. We notice that for these types of meshes, the numerical solutions obtained using KR velocities are essentially the same as those obtained using C velocities. 

\begin{figure}[h]
	\centering
	\begin{tabular}{c@{\hspace*{2em}}c}
		\includegraphics[width=0.45\textwidth]{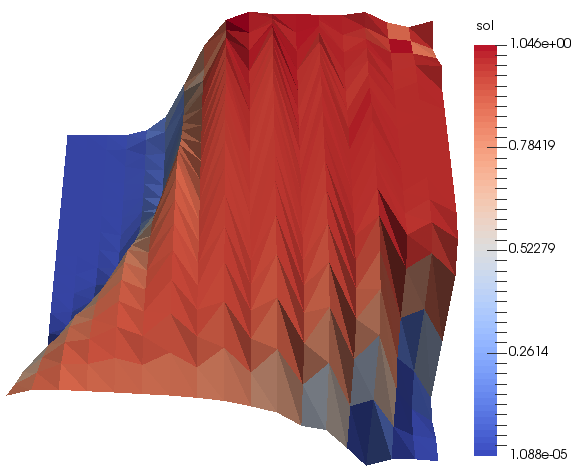} & 		\includegraphics[width=0.45\textwidth]{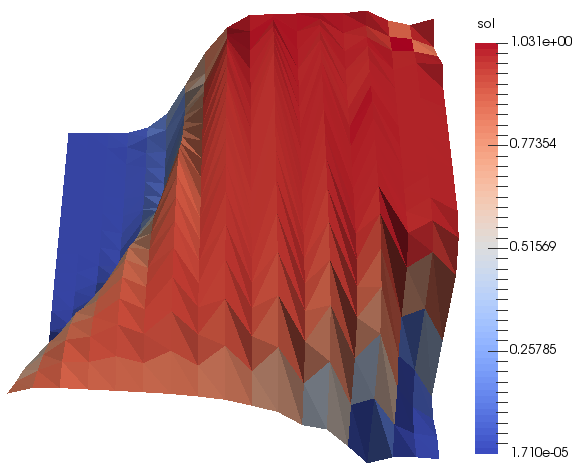}
	\end{tabular}
	\caption{Kershaw mesh, 6 points per edge, weighted trapezoid rule for source terms, $t=10$ years (left: KR velocities, right: C velocities).}
	\label{CK310}
\end{figure}

Now, looking at the numerical results along Kershaw type meshes on Figure \ref{CK310}, we see that C velocities give us a slightly better numerical result. The maximum overshoot is now around $ 3.1 \%$, as opposed to a maximum overshoot of around $ 4.5 \%$ for KR velocities. Moreover, we observe that the solutions on the Cartesian, hexahedral, and non-conforming meshes are very similar, showing a certain robustness of the method with respect to the choice of mesh. However, the solution on the Kershaw mesh is noticeably different, which signals the presence of a grid effect. In the next section, we consider streamlines to back up
our claim that C velocities yield a better numerical approximation than KR velocities. These streamlines also enable us to understand why grid effects are present in Kershaw type meshes.

\subsection{Streamlines resulting from the reconstructed Darcy velocities}\label{sec:stream}

We start by plotting the streamlines for the velocities reconstructed using both KR and C velocities on Cartesian, hexahedral, and non-conforming meshes, which can be seen in figures \ref{SCart} to \ref{SnCon}. For the following figures, the particles are assumed to have travelled for 3600 days, which is approximately 10 years.
\begin{figure}[h]
	\centering
	\begin{tabular}{c@{\hspace*{2em}}c}
		\includegraphics[width=0.45\textwidth]{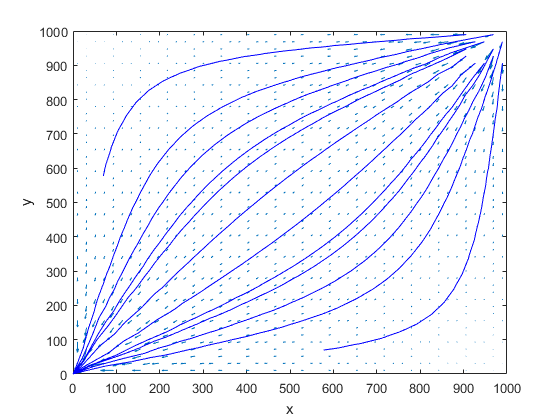} & 		\includegraphics[width=0.45\textwidth]{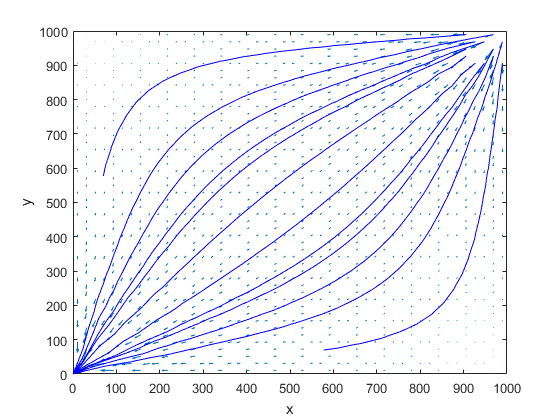}
	\end{tabular}
	\caption{Streamlines along Cartesian mesh, (left: KR velocities; right: C velocities).}
	\label{SCart}
\end{figure}
\begin{figure}[h]
	\centering
	\begin{tabular}{c@{\hspace*{2em}}c}
		\includegraphics[width=0.45\textwidth]{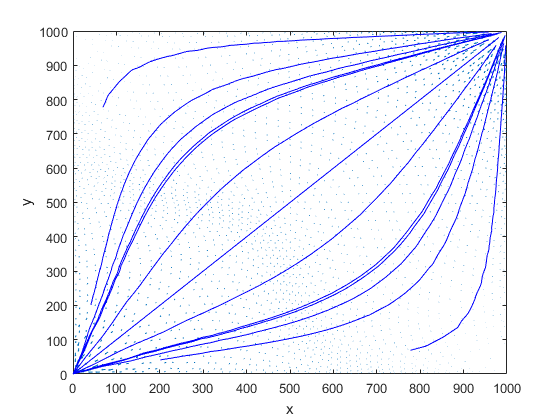} & 		\includegraphics[width=0.45\textwidth]{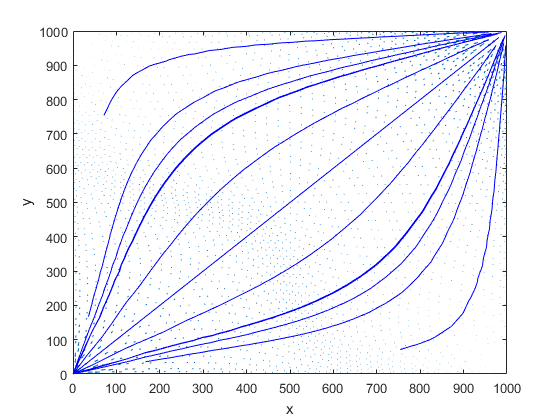}
	\end{tabular}
	\caption{Streamlines along the hexahedral mesh, (left: KR velocities; right: C velocities).}
	\label{SHex}
\end{figure}
\begin{figure}[h]
	\centering
	\begin{tabular}{c@{\hspace*{2em}}c}
		\includegraphics[width=0.45\textwidth]{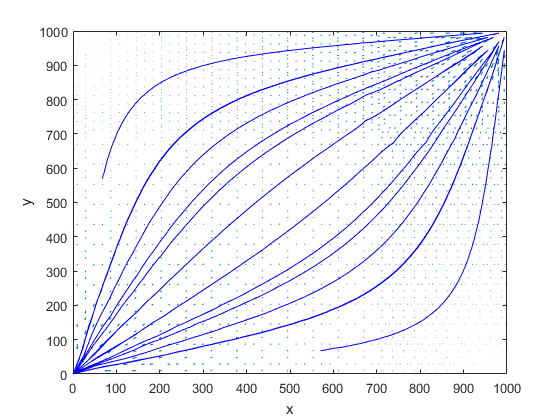} & 		\includegraphics[width=0.45\textwidth]{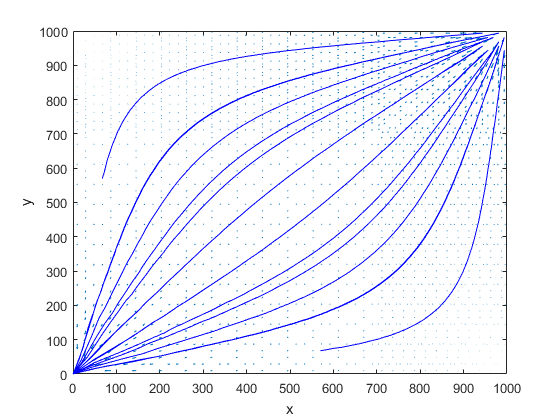}
	\end{tabular}
	\caption{Streamlines along the non-conforming mesh, (left: KR velocities; right: C velocities).}
	\label{SnCon}
\end{figure}

As can be seen in figures \ref{SCart} to \ref{SnCon}, the streamlines along the Cartesian, hexahedral, and non-conforming meshes are quite similar (whether we use KR velocities or C velocities), which explains why the numerical solutions obtained for the concentration are essentially the same, regardless of which velocity we use. We note however, that upon comparing the streamlines of hexahedral type meshes against those of Cartesian or non-conforming meshes, that there is a slight difference in how the fluid travels. In particular, we take note that the streamline arising from the rightmost position of the plots ends near position (770, 60) for the hexahedral meshes, and near position (570, 60) for the other 2 meshes. This particular difference can also be seen in the concentration profiles upon comparing Figure \ref{H310} with Figures \ref{CC310} and \ref{nCon_l2}. Note here that the concentration profile obtained on a hexahedral mesh exhibits a sharper fingering effect along the diagonal, as compared to the concentration profile obtained on the other 2 meshes. This phenomenon is caused by using fluxes generated by the low-order HMM method, which are prone to grid effects; in a future work, we will be exploring the usage, for the pressure equation only, of high-order methods such as the HHO scheme \cite{Di-Pietro.Ern.ea:14} to improve the quality of the numerical fluxes on distorted meshes. The effect is not as severe as the Kershaw type meshes though, since hexahedral meshes are not as distorted. 
\begin{figure}[h]
	\centering
	\begin{tabular}{c@{\hspace*{2em}}c}
		\includegraphics[width=0.45\textwidth]{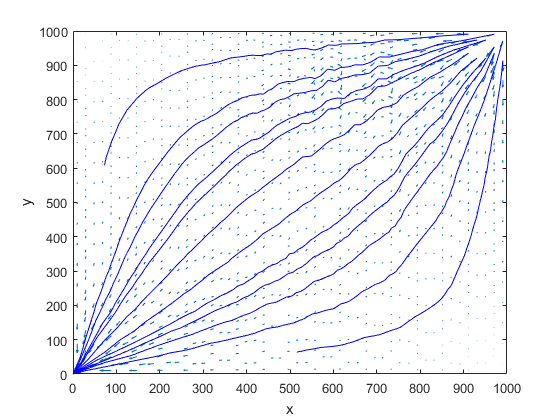} & 		\includegraphics[width=0.45\textwidth]{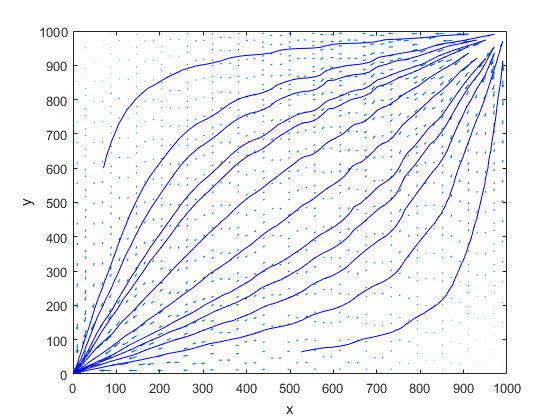}
	\end{tabular}
	\caption{Streamlines along Kershaw mesh, (left: KR velocities; right: C velocities).}
	\label{SKer}
\end{figure}

The streamlines presented in Figure \ref{SKer} show that, on Kershaw meshes, the flow resulting from the C velocities is
more natural than that of KR velocities, which exhibits more staggering in the upper region. This explains why the numerical solution for the concentration obtained using the C velocities is slightly better than that of the KR velocities. Noting that the streamlines obtained along all 4 types of meshes are quite similar, we have still yet to address the grid effects that were present for the Kershaw mesh. We understand this by tracking the particles in the streamline for a shorter time period of 2520 days, or approximately 7 years. For Figures \ref{Sch7yr} and \ref{Snk7yr}, we will be using C velocities.
\begin{figure}[h]
	\centering
	\begin{tabular}{c@{\hspace*{2em}}c}
		\includegraphics[width=0.45\textwidth]{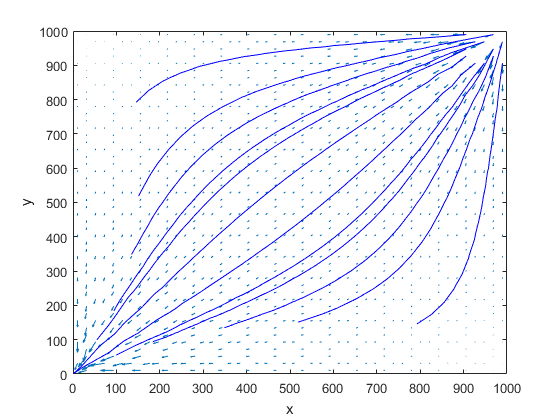} & 		\includegraphics[width=0.45\textwidth]{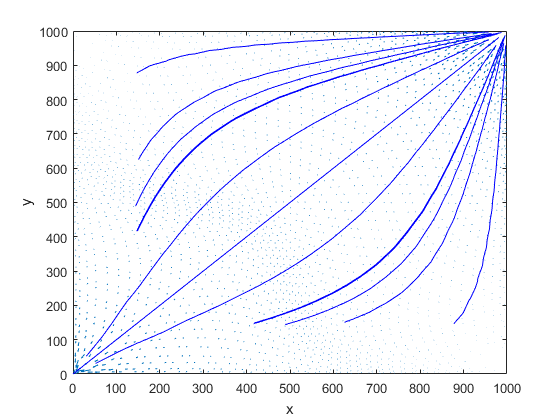}
	\end{tabular}
	\caption{Streamlines at 2520 days, C velocities (left: Cartesian mesh; right: Hexahedral mesh).}
	\label{Sch7yr}
\end{figure}

\begin{figure}[h]
	\centering
	\begin{tabular}{c@{\hspace*{2em}}c}
		\includegraphics[width=0.45\textwidth]{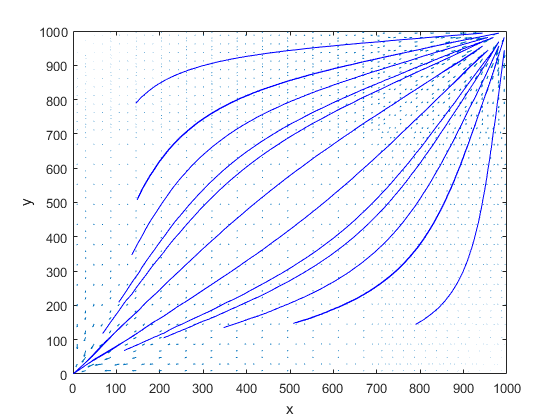} & 		\includegraphics[width=0.45\textwidth]{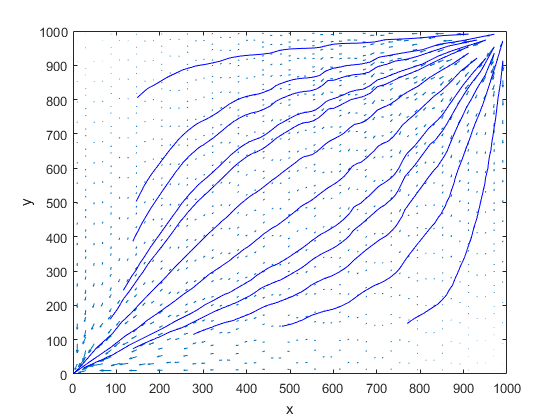}
	\end{tabular}
	\caption{Streamlines at 2520 days, C velocities (left: non-conforming mesh; right: Kershaw mesh).}
	\label{Snk7yr}
\end{figure}

As can be seen in Figure \ref{Sch7yr}, the streamlines resulting from both Cartesian and hexahedral meshes are almost symmetric with respect to the line $y=x$. A similar observation can be made for non-conforming meshes (see Fig. \ref{Snk7yr} left). On the contrary, due to the large distortion of the Kershaw mesh (see Fig. \ref{NKmeshes} right), the advection field on this mesh is such that particles travelling below the line $y=x$ reach the production well $(0,0)$ faster than those travelling above the line. In particular, upon looking at the plot on the right of Figure \ref{Snk7yr}, we focus on the third streamlines from the right and top boundaries, which should be symmetric about the diagonal $x=y$. 
The streamline below $y=x$ has travelled near the point (275,100), whereas the streamline
above the line $y=x$ has only reached some point near (150,400). Hence, the distorted mesh leads to
an advection of the fluid that is skewed towards
the lower part of the domain, thus leading to the numerical results that vary from those obtained in the other types of meshes. 
We may also compare this third streamlines to the streamlines obtained from the other types of meshes. This comparison confirms that, for Kershaw meshes, advection below the line $y=x$ is much faster than expected. Similar observations are obtained upon looking at streamlines obtained from KR velocities. 
\begin{figure}[h]
	\centering
	\begin{tabular}{c@{\hspace*{2em}}c}
		\includegraphics[width=0.45\textwidth]{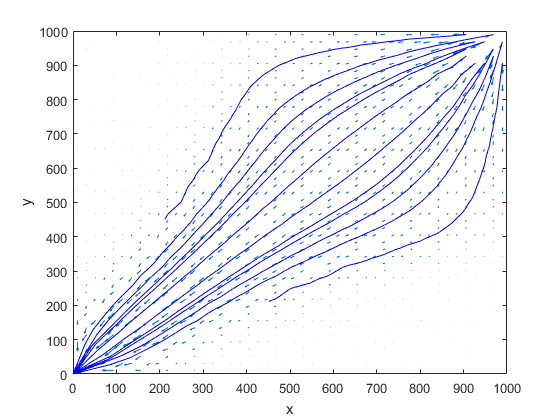} & 		\includegraphics[width=0.45\textwidth]{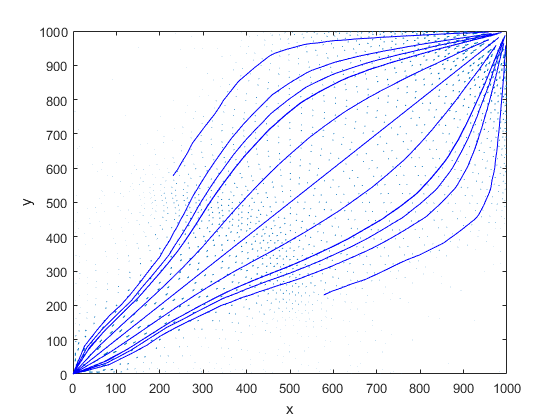}
	\end{tabular}
	\caption{Streamlines using velocity profile at final time step, C velocities (left: Cartesian mesh; right: Hexahedral mesh).}
	\label{Schyr10}
\end{figure}

\begin{figure}[h]
	\centering
	\begin{tabular}{c@{\hspace*{2em}}c}
		\includegraphics[width=0.45\textwidth]{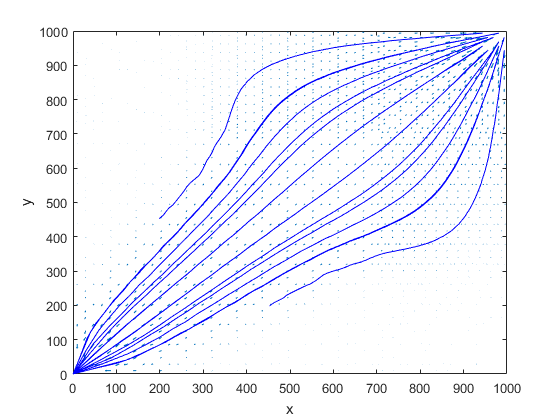} & 		\includegraphics[width=0.45\textwidth]{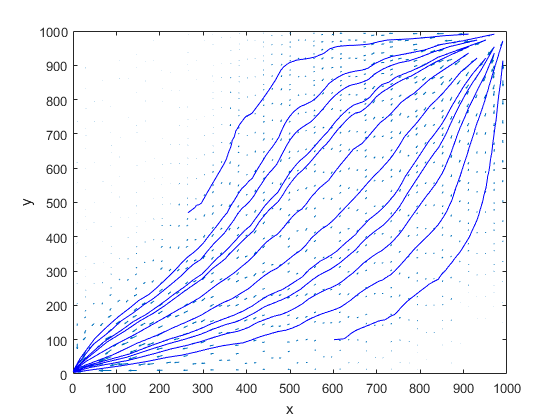}
	\end{tabular}
	\caption{Streamlines using velocity profile at final time step, C velocities (left: non-conforming mesh; right: Kershaw mesh).}
	\label{Snkyr10}
\end{figure}

\begin{remark}
	Figures \ref{SCart} to \ref{Snk7yr} were obtained from the velocity profile at the first time step; hence, the dependency of the velocity profile on the concentration $c$ due to a high mobility ratio of $M=41$ was not visible. To complete the presentation, we show in Figures \ref{Schyr10} and \ref{Snkyr10} the velocity profile obtained at the final time step. Also, the particles along the streamline are assumed to have traveled 3600 days, or approximately 10 years. Indeed, upon looking at these figures side by side with Figures \ref{CC310} to \ref{CK310}, we see the dependence of the velocity profile on the concentration, i.e., it tends to flow along the region(s) with high concentration. 
\end{remark}

\subsection{Comparison with other numerical schemes}

In this section, we compare the numerical results obtained from HMM--ELLAM to numerical results obtained from other schemes, such as MFV with upwinding \cite{CD-07} and MFEM--ELLAM \cite{WLELQ-00}. This will be done on two test cases. The first test case will be done under the same test data and parameters considered above. The second test case will be done instead with an inhomogeneous permeability tensor $\mathbf{K}=20\mathbf{I}$ md over the region $(200,400)\times(200,400)\cup (200,400)\times(600,800)\cup(600,800)\times(200,400)\cup(600,800)\times(600,800)$  and $\mathbf{K}=80\mathbf{I}$ md elsewhere, while holding all other test data and parameters to be the same as those of the initial test case. This comparison is performed on a Cartesian mesh of size $50 \times 50$ ft (so that the discontinuities present in the second test case are aligned with the edges of the cells); other meshes
could be considered (triangular for MFEM--ELLAM, and any polytopal mesh for MFV--upwind), with
similar conclusions. We perform the second test case with a time step of $\Delta t=2.5$ days, as opposed to the time step of $36$ days for the initial test case. We do this in order to have a fair comparison with the higher order scheme implemented in \cite{AD17} (which was second order in time, implemented with a time step of $\Delta t=7.2$ days). As a point of reference, we present in Figures \ref{CC10-contour} and \ref{nonhomo-CC10-contour} the concentration profile and contour plot at $t=10$ years for the HMM--ELLAM, for the first and second test cases, respectively.  
\begin{figure}[h]
	\begin{tabular}{cc}
		\includegraphics[width=0.4\linewidth]{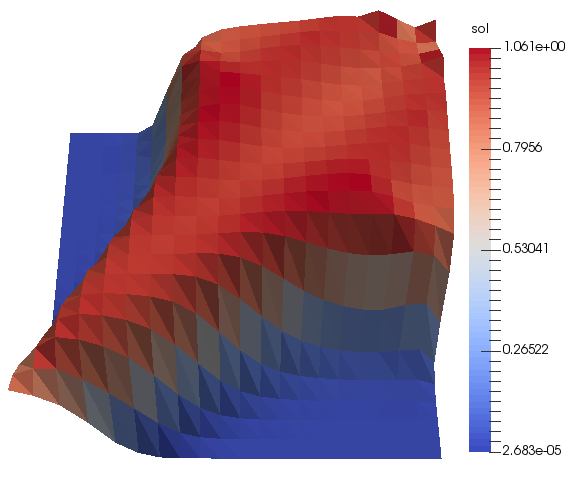} &
		\includegraphics[width=0.4\linewidth]{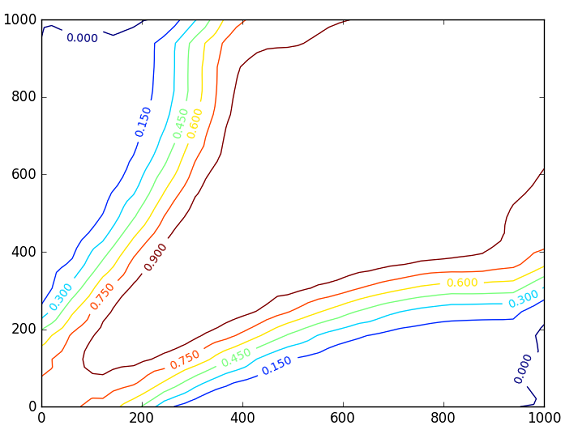} 		
	\end{tabular}
	\caption{Numerical concentration obtained through HMM--ELLAM at $t=10$ years, homogeneous permeability (left: profile; right: contour plot).}
	\label{CC10-contour}
\end{figure}
\begin{figure}[h]
	\begin{tabular}{cc}
		\includegraphics[width=0.4\linewidth]{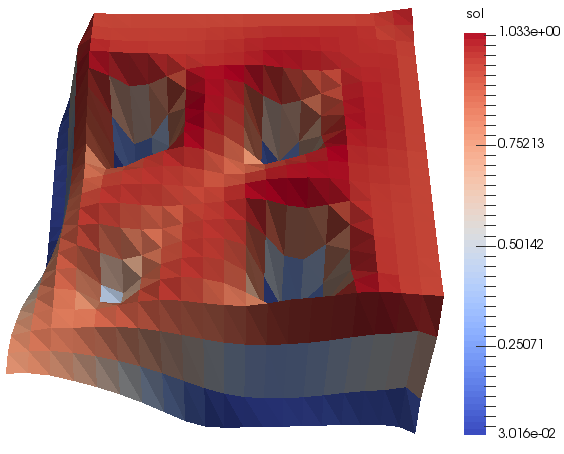} &
		\includegraphics[width=0.4\linewidth]{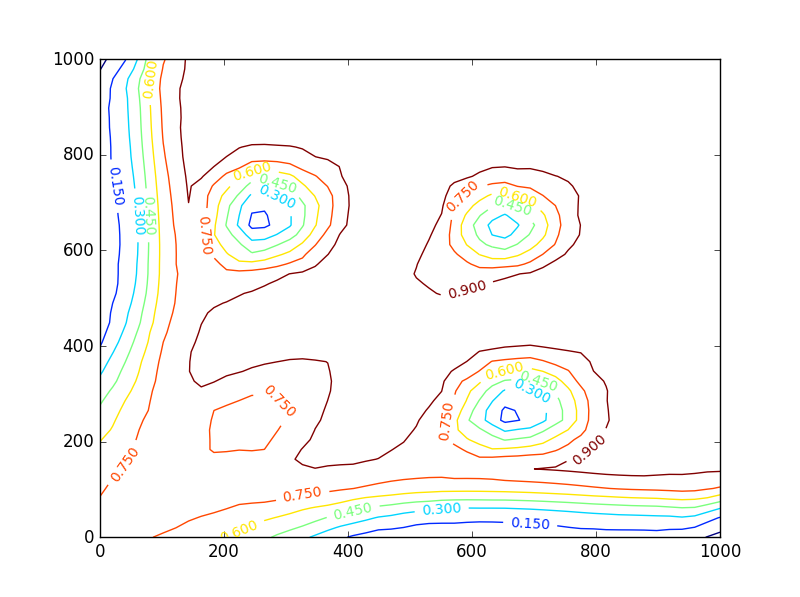} 		
	\end{tabular}
	\caption{Numerical concentration obtained through HMM--ELLAM at $t=10$ years, inhomogeneous permeability (left: profile; right: contour plot).}
	\label{nonhomo-CC10-contour}
\end{figure}
\subsubsection{MFEM--ELLAM}
As both MFEM--ELLAM and HMM--ELLAM are characteristic methods, tracking is implemented for both schemes for the concentration equation. Typically, HMM--ELLAM schemes only need to track the vertices, together with 1 point per edge for Cartesian type meshes, unless the time step is too large relative to the spacial discretisation (which will result to either a degenerate or self intersecting polygon approximating the trace back region). This can be avoided by reducing the time step or increasing the number of points per edge. However, MFEM--ELLAM schemes need to track a bare minimum of 3-4 points in each cell to get a correct quadrature rule to integrate the basis functions (and much more than 4 points in case these bases functions become too distorted by the tracking velocity \cite{S15}); this is why we expect the MFEM--ELLAM to be computationally more expensive than the HMM--ELLAM.

 Aside from computational cost, a more important thing to consider would be the quality of the numerical solutions. Figures \ref{MFEM--ELLAM-10} and \ref{nonhomo-MFEM--ELLAM-10} give us the numerical solution and contour plot obtained from MFEM--ELLAM at $t=10$ years for the first and second test cases, respectively. These numerical outputs were obtained by a straight application of the MFEM--ELLAM algorithm as presented in \cite{WLELQ-00}, with several hundred
of quadrature points per cell around the injection well.

 \begin{figure}[h]
 	\centering
 	\begin{tabular}{c@{\hspace*{2em}}c}
 		\includegraphics[width=0.4\textwidth]{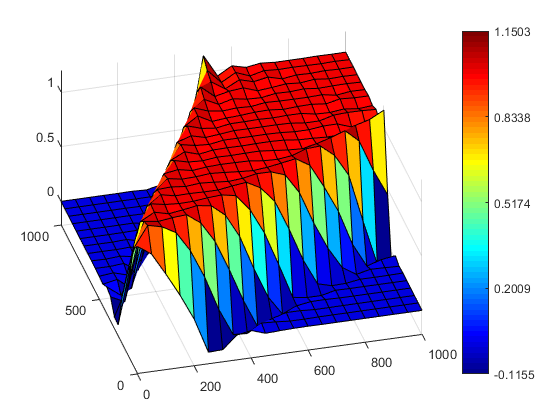} & 		\includegraphics[width=0.4\textwidth]{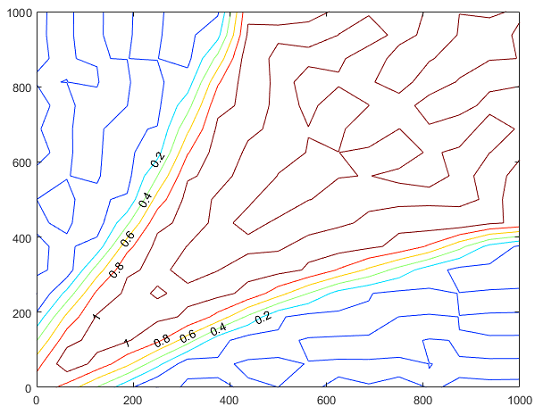}
 	\end{tabular}
 	\caption{ Numerical concentration obtained through MFEM--ELLAM at $t=10$ years, homogeneous permeability (left: profile; right: contour plot).}
 	\label{MFEM--ELLAM-10}
 \end{figure}
 \begin{figure}[h]
	\centering
	\begin{tabular}{c@{\hspace*{2em}}c}
		\includegraphics[width=0.4\textwidth]{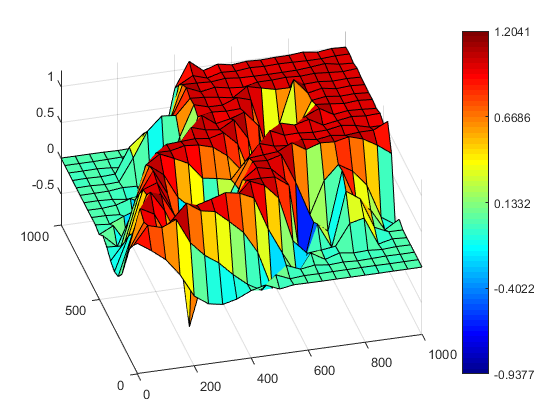} & 		\includegraphics[width=0.4\textwidth]{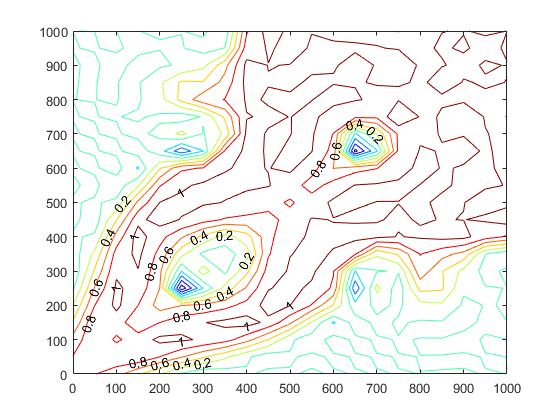}
	\end{tabular}
	\caption{ Numerical concentration obtained through MFEM--ELLAM at $t=10$ years, inhomogeneous permeability (left: profile; right: contour plot).}
	\label{nonhomo-MFEM--ELLAM-10}
\end{figure}
The concentration profile and the contour plots obtained from both schemes are quite similar, with overshoots and undershoots that are typical of characteristic-based methods \cite{ERW83}. However, here, the overshoot in the MFEM--ELLAM (around 15\% for homogeneous permeability and 20\% for inhomogeneous permeability) is higher than that of the HMM--ELLAM (6\% and 3\% respectively) for both cases. Note that for the source terms, the MFEM--ELLAM integrates a non-constant function through quadrature rules. The main source of error encountered upon computing these integrals arise due to the presence of steep source terms. However, for the HMM--ELLAM, a different treatment of the source terms (see Section \ref{numSource}) was implemented. We recall that this can be physically interpreted as spreading the injected fluid around the region surrounding the injection well.  Naturally, this resulted to the HMM--ELLAM having a smaller overshoot than the MFEM-ELLAM. We also note that the MFEM--ELLAM exhibits severe undershoots (up to around 11\% for homogeneous permeability near the production cell and 90\% for inhomogeneous permeability in the low-permeability regions), which is not present in the HMM--ELLAM. This severe undershoot might be due to the fact that conforming FE methods, used for solving the concentration equation, have unknowns on the vertices that sit at the permeability discontinuities. It has been noticed that, for transport of species in heterogeneous domains, schemes with unknowns at the vertices may lead to
unacceptable results on coarse meshes, see \cite{EGHM12}. On the other hand, the transition layer (from $c\approx 0$ to $c\approx 1$) is thinner for the MFEM--ELLAM than for the HMM--ELLAM.

\subsubsection{MFV--upwind}
Over each time step, the MFV--upwind requires, for the concentration, the solution of a linear system which has the same sparsity and number of unknowns as the HMM--ELLAM. Moreover, due to the absence of characteristic tracking and computation of integrals over trace-back regions, the computational cost of MFV-upwind scheme is much cheaper than that of the HMM--ELLAM. Next, we compare the quality of the solutions obtained by looking at Figures \ref{C10-upwind} and \ref{noncon-C10-upwind}. It is quite notable that  the solution remains bounded between 0 and 1 (actually, no undershoot occurs, and the overshoot is less than 0.01\%).
 However, upwind schemes tend to introduce excess diffusion, and thus the strong viscous fingering effects we expect have been spread out. This can be seen more clearly by looking at the contour plot (Figure \ref{C10-upwind} right). Upon comparison with contour plots obtained for the HMM--ELLAM and MFEM--ELLAM schemes (Figures  \ref{CC10-contour} and \ref{MFEM--ELLAM-10}, right), we indeed see that the strong viscous fingering expected along the diagonal has been spread out by the upwind scheme. A similar conclusion can be drawn for the inhomogeneous permeability tensors upon comparing Figure \ref{noncon-C10-upwind} to Figures \ref{nonhomo-CC10-contour} and \ref{nonhomo-MFEM--ELLAM-10}. However, upon comparing Figures \ref{nonhomo-CC10-contour}, \ref{nonhomo-MFEM--ELLAM-10}, and \ref{noncon-C10-upwind} we notice that the concentration profiles obtained on both HMM--ELLAM and MFV-upwind are quite similar. Even the transition layers obtained on the respective contour plots seem to agree with one another. Due to this, we cannot make a definitive conclusion at this point. Another point of comparison will be presented in Section \ref{sec:rec-oil}. 
\begin{figure}[h]
	\begin{tabular}{cc}
		\includegraphics[width=0.4\linewidth]{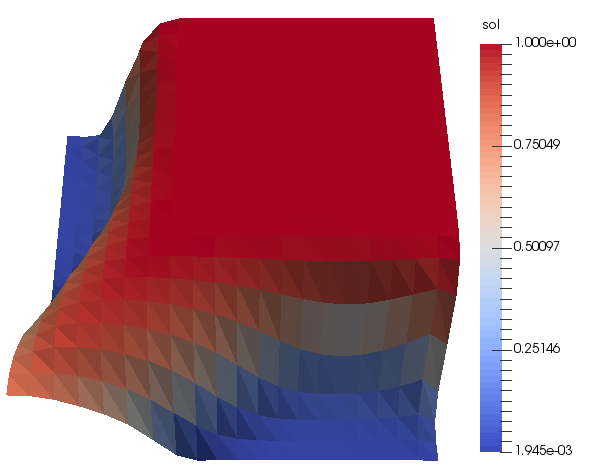} & 		\includegraphics[width=0.4\linewidth]{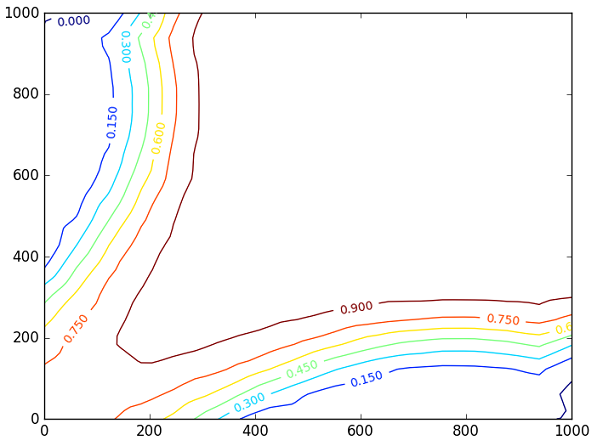}
	\end{tabular}
	\caption{ Numerical concentration obtained through MFV--upwinding at $t=10$ years, homogeneous permeability (left: profile; right: contour plot) }
	\label{C10-upwind}
\end{figure}
\begin{figure}[h]
	\begin{tabular}{cc}
		\includegraphics[width=0.4\linewidth]{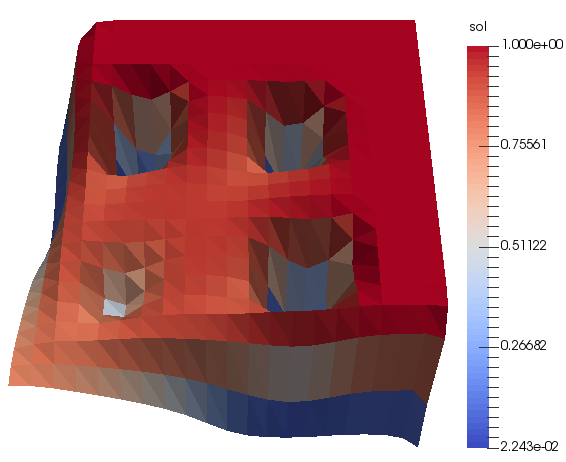} & 		\includegraphics[width=0.4\linewidth]{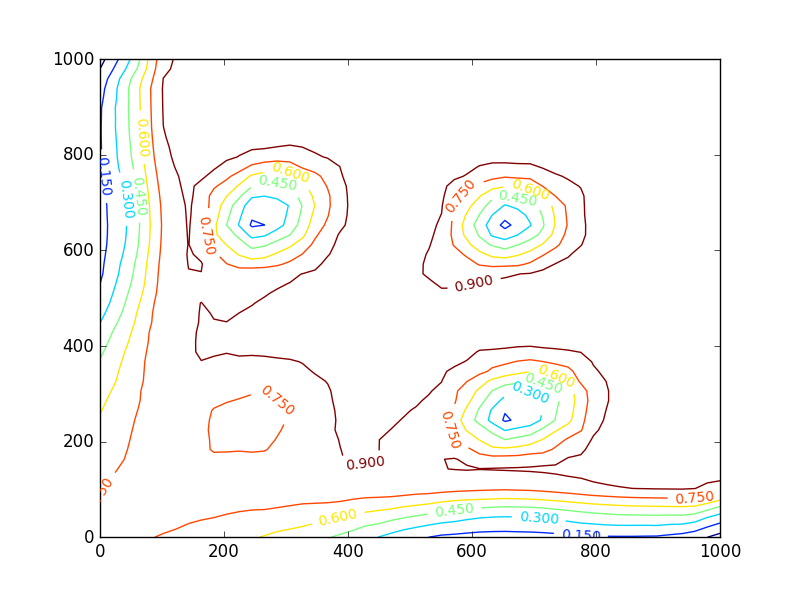}
	\end{tabular}
	\caption{ Numerical concentration obtained through MFV--upwinding at $t=10$ years, inhomogeneous permeability (left: profile; right: contour plot) }
	\label{noncon-C10-upwind}
\end{figure}

\subsubsection{Recovered oil} \label{sec:rec-oil}

A particular quantity of interest in performing numerical simulations of the model \eqref{eq:model} is the amount of oil (percentage of domain) recovered at time $T$, given by $|\O|^{-1} \int_\O \phi c(x,T)$. Figure \ref{oil_recovered} shows the percentage of domain recovered at $t=10$ years, for each of the three numerical schemes on Cartesian meshes. These results were obtained starting with a very coarse mesh consisting of square cells of dimension $200\times 200$, while refining the spacial discretisation by a factor of 2, leading to a final mesh consisting of square cells of dimension $25\times 25$.
\begin{figure}[h]
			\begin{tabular}{cc}
			\includegraphics[width=0.48\linewidth]{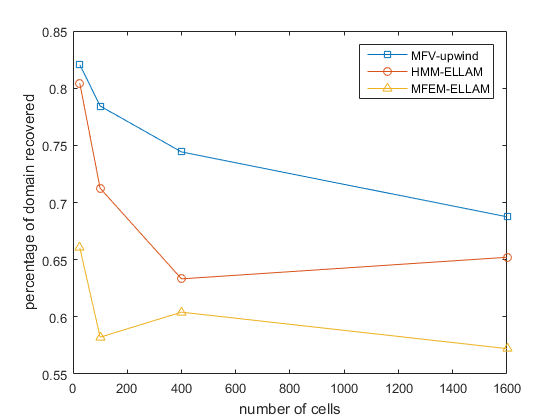} & 		\includegraphics[width=0.48\linewidth]{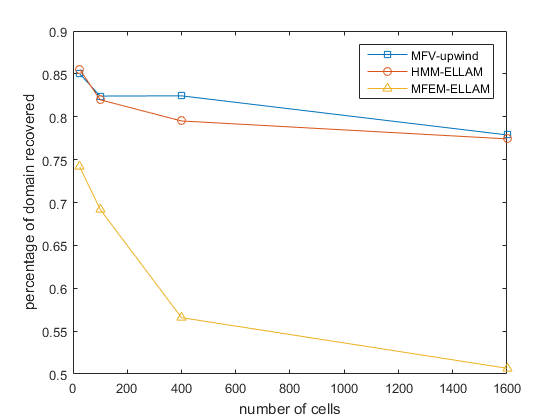}
		\end{tabular}	
	\caption{ Percentage of domain recovered at $t=10$ years (left: homogeneous permeability; right: inhomogeneous permeability).}
	\label{oil_recovered}
\end{figure}

It can be seen here that each of the three schemes seem to yield numerical results that converge to different quantities. The MFV--upwind scheme probably overestimates the amount of oil recovered due to the diffusion it introduces, and thus a "wider" region has been sweeped. On the other hand, it is likely that the MFEM--ELLAM underestimates the amount of oil recovered due to the presence of undershoots, as we have noted from Figures \ref{MFEM--ELLAM-10} and \ref{nonhomo-MFEM--ELLAM-10} (left). Thus, we expect the amount of oil recovered to be somewhere between these two values. Here, the solution obtained from HMM--ELLAM converges to such a value, which seems to sit about 65\% for the first test case and 77\% for the second test case. 

As an element of comparison, we consider the results obtained in \cite{AD17} on the same model with the Hybrid High-Order (HHO) method. This method is based on polynomials, with an arbitrary chosen degree, in the cells and on the faces, and can theoretically achieve any order of accuracy on any polytopal mesh (at an increased computational cost compared to HMM and ELLAM methods, of course); in practice, though, tests are usually ran up to order 4 or 5; we refer to \cite{Di-Pietro.Ern.ea:14} for the presentation on the pure diffusion equation, and to \cite{DDE15} for advection--diffusion--reaction models. The tests in \cite{AD17} were performed up to order 4 and seem to indicate that the expected recovery after 10 years is around 65\% for the homogeneous test case, and around 75\% for the inhomogeneous test case. The low-order, less expensive HMM--ELLAM method seems to provide very similar results, on the contrary to MFV--ELLAM and MFEM--ELLAM. The latter, in particular, only predicts a 50\% recovery for the inhomogeneous test case, which is much lower than all other methods; this is naturally expected because of the presence of severe undershoots. The concentration profiles obtained from the HMM--ELLAM are also similar to the one obtained through the HHO in \cite{AD17}.

Moreover, due to the thinner transition layer present in the HMM--ELLAM, this solution is preferred over the MFV-upwind scheme. On the other hand, since the solutions between the HMM--ELLAM and MFV-upwind are quite close to each other for the second test case, one might argue that MFV-upwind would be preferable, due to the cheaper computational cost associated with it. This is not necessarily the case, as we note that the HMM--ELLAM already generates an acceptable concentration profile, as well as a good approximation of the amount of oil recovered upon taking a mesh with square cells of dimension $50\times 50$ ft, whereas MFV-upwind needs to go to the refined mesh with square cells of dimension $25\times 25$ ft (which requires solving a much larger system) before producing such results.

\section{Conclusion}

Most of the previous work focused on separately developing  schemes for either solutions of anisotropic diffusion equations (corresponding to the pressure equation \eqref{pressure}) or schemes for advection dominated equations (corresponding to the concentration equation \eqref{concentration}).
This work presents a complete scheme for the coupled system, with details on how to combine the scheme for the pressure equation (HMM) with the scheme for the concentration equation (ELLAM). Moreover, this is usable on generic meshes as encountered in real-world applications -- with the usual caveats on distorted meshes. 

A new method has been introduced to reconstruct Darcy velocities from internal numerical fluxes, based
on a consistency relation. This method turns out to be more cost efficient, and more efficient on distorted meshes than the one presented in \cite{KR03-kuznetsov-repin}.
Its construction in the 2-dimensional case seems promising enough to envision developing a 3-dimensional
version, which is the purpose of ongoing work.

Our analysis also demonstrates the importance of selecting the correct number of approximation points -- depending on the regularity of the mesh -- to track the cells. An improved treatment of the cells near injection wells has also been introduced, leading to less overshoots and better approximations.

The streamlines in Section \ref{sec:stream} show that velocities reconstructed from a piecewise constant approximation of the edge fluxes exhibit grid effects, especially on distorted grids. A natural direction for the improvement of the scheme, currently being explored, is to solve the pressure equation \eqref{pressure} with a slightly higher order method. High order schemes for anisotropic diffusion equations have been developed for generic meshes, and a generic framework which covers a number of these schemes has been presented in \cite{DSGD18-High-order-methods}. It is therefore possible, on distorted meshes, to implement one of these high order methods for the pressure equation only, with the hope that it provides more accurate fluxes than the HMM to be used in the reconstruction of the Darcy velocity for the ELLAM component on the concentration equation. Since the high order scheme would only be used on the pressure equation, the overall computational cost would be much reduced compared to the full HHO method in \cite{AD17}.

Upon comparison with MFV--upwinding and MFEM--ELLAM schemes, we also got to see that the HMM--ELLAM captures the concentration profile better than upwinding, although as compared to MFEM--ELLAM we have a larger transition layer, as presented in the contour plots in Figures \ref{CC10-contour} to \ref{noncon-C10-upwind} (right). As for the amount of oil recovered, it seems that the HMM--ELLAM also performs better than both MFEM--ELLAM and MFV--upwind schemes for both test cases. Hence, overall, an HMM--ELLAM scheme is preferred over MFEM--ELLAM or MFV-upwind schemes.

\bigskip

\thanks{\textbf{Acknowledgement}: this research was supported by the Australian Government through the Australian Research Council's Discovery Projects funding scheme (pro\-ject number DP170100605).
}

%

\bibliographystyle{abbrv}
\bibliography{hmm-ellam}
\end{document}

%% file: macros.tex
\allowdisplaybreaks

\definecolor{labelkey}{rgb}{0.6,0,1}

\usepackage[normalem]{ulem}
\usepackage{bm}
\definecolor{violet}{rgb}{0.580,0.,0.827}
\def\corr#1#2#3{\typeout{Warning : a correction remains in page
\thepage}
        {\color{blue}{\ifmmode\text{\,\sout{\ensuremath{#1}}\,}\else\sout{#1}\fi}}%
        \textcolor{violet}{#2}
        \textcolor{red}{#3}}

\newcounter{cst}
\catcode`\@=11
\def\ctel#1{C_{\refstepcounter{cst}\@bsphack
\protected@write\@auxout{}%
           {\string\newlabel{#1}{{\thecst}{\thepage}}}\thecst}}
\catcode`\@=12

\newcounter{cexp}
\catcode`\@=11
\def\terml#1{T_{\refstepcounter{cexp}\@bsphack
\protected@write\@auxout{}%
           {\string\newlabel{#1}{{\thecexp}{\thepage}}}\thecexp}}
\catcode`\@=12

\newcommand{\mathbi}[1]{{\boldsymbol #1}}

\newcommand{\eop}{{\unskip\nobreak\hfil\penalty50
           \hskip2em\hbox{}\nobreak\hfil\mbox{\rule{1ex}{1ex} \qquad}
   \parfillskip=0pt
   \finalhyphendemerits=0\par\medskip}}
\renewenvironment{proof}[1][]{\noindent {\bf Proof#1. } }{\eop}

\newtheorem{theorem}{Theorem}[section]
\newtheorem{remark}[theorem]{Remark}
\newtheorem{lemma}[theorem]{Lemma} 

\definecolor{shadecolor}{gray}{0.92}
\definecolor{TFFrameColor}{gray}{0.92}
\definecolor{TFTitleColor}{rgb}{0,0,0}

\newcommand{\ba}{\begin{array}{llll}   }
\newcommand{\bac}{\begin{array}{c}}
\newcommand{\bari}{\begin{array}{r}}
\newcommand{\ea}{\end{array}}

\newcommand{\ban}{\begin{array}{llll}}
\newcommand{\ean}{\end{array}}

\newcommand{\be}{\begin{equation}}
\newcommand{\ee}{\end{equation}}

\newcommand{\beqsys }{\beqtab \left \{ \begin{array}{l}}
\newcommand{\eeqsys }{\end{array} \right . \eeqtab }

\newcommand{\benum}{\begin{enumerate}}
\newcommand{\eenum}{\end{enumerate}}

\newcommand{\beqtab}{\begin{eqnarray}} 
\newcommand{\eeqtab}{\end{eqnarray}}

\newcommand{\bfn}{\mathbf{n}}

\newcommand{\centercv}{\x_\cv}                         
\newcommand{\centers}{\mathcal{P}}
\newcommand{\cv}{K}

\renewcommand{\d}{{\rm d}}
\newcommand{\darcyU}{\mathbf{u}}
\newcommand{\dcvedge}{d_{\cv,\edge}}

\newcommand{\disc}{{\mathcal D}}

\renewcommand{\div}{{\mathop{\rm div}}}
\newcommand{\proj}{\mathcal{P}}

\newcommand{\edge}{\sigma}

\newcommand{\edges}{{\mathcal E}}              
\newcommand{\edgescv}{{{\edges}_\cv}}

\newcommand{\grad}{\nabla}

\newcommand{\mesh}{{\mathcal M}}

\newcommand{\ncvedge }{\bfn_{\cv,\edge}}  

\renewcommand{\O}{\Omega}

\newcommand{\points}{\centers}

\newcommand{\R}{\mathbb R}

\def\v{\mathbf{v}}

\newcommand{\x}{\mathbf{x}}

\newcommand{\centeredge}{\overline{\mathbi{x}}_\edge} 

\def\ograd{\overline{\grad}}

\usepackage{xparse}
\DeclareDocumentCommand{\RPiD}{ O{\disc} O{,0} }{\Pi_{#1}(X_{#1#2})}

\def\Fdof#1{{\bm{\mathcal{F}}(#1,\R)}}
\makeatletter
\def\Fdof{\@ifnextchar[{\@with}{\@without}}
\def\@with[#1]#2{{\bm{\mathcal{F}}(#2;#1)}}
\def\@without#1{{\bm{\mathcal{F}}(#1,\R)}}
\makeatother

\def\RT0{\mathbb{RT}_0}

%% file: fig-subd.pdf_t
\begin{picture}(0,0)%
\includegraphics{fig-subd.pdf}%
\end{picture}%
\setlength{\unitlength}{4144sp}%
\begingroup\makeatletter\ifx\SetFigFont\undefined%
\gdef\SetFigFont#1#2#3#4#5{%
  \reset@font\fontsize{#1}{#2pt}%
  \fontfamily{#3}\fontseries{#4}\fontshape{#5}%
  \selectfont}%
\fi\endgroup%
\begin{picture}(2456,1717)(368,-1160)
\put(1656,  6){\makebox(0,0)[lb]{\smash{{\SetFigFont{10}{12.0}{\rmdefault}{\mddefault}{\updefault}{\color[rgb]{0,0,0}$\x_K$}%
}}}}
\put(2010,-10){\makebox(0,0)[lb]{\smash{{\SetFigFont{10}{12.0}{\rmdefault}{\mddefault}{\updefault}{\color[rgb]{0,0,0}$F_{\sigma^*_i}$}%
}}}}
\put(2619,-802){\makebox(0,0)[lb]{\smash{{\SetFigFont{10}{12.0}{\rmdefault}{\mddefault}{\updefault}{\color[rgb]{0,0,0}$T_{K,\sigma_i}$}%
}}}}
\put(1888,-994){\makebox(0,0)[lb]{\smash{{\SetFigFont{10}{12.0}{\rmdefault}{\mddefault}{\updefault}{\color[rgb]{0,0,0}$F_{K,\sigma_i}$}%
}}}}
\put(2611,-556){\makebox(0,0)[lb]{\smash{{\SetFigFont{10}{12.0}{\rmdefault}{\mddefault}{\updefault}{\color[rgb]{0,0,0}$\v_{i}$}%
}}}}
\put(991,-556){\makebox(0,0)[lb]{\smash{{\SetFigFont{10}{12.0}{\rmdefault}{\mddefault}{\updefault}{\color[rgb]{0,0,0}${\sigma^{*}_{i-1}}$}%
}}}}
\put(2206,-331){\makebox(0,0)[lb]{\smash{{\SetFigFont{10}{12.0}{\rmdefault}{\mddefault}{\updefault}{\color[rgb]{0,0,0}${\sigma^*_i}$}%
}}}}
\put(946,-1096){\makebox(0,0)[lb]{\smash{{\SetFigFont{10}{12.0}{\rmdefault}{\mddefault}{\updefault}{\color[rgb]{0,0,0}$\v_{i-1}$}%
}}}}
\put(1396,-961){\makebox(0,0)[lb]{\smash{{\SetFigFont{10}{12.0}{\rmdefault}{\mddefault}{\updefault}{\color[rgb]{0,0,0}${\sigma_i}$}%
}}}}
\put(1621,-556){\makebox(0,0)[lb]{\smash{{\SetFigFont{10}{12.0}{\rmdefault}{\mddefault}{\updefault}{\color[rgb]{0,0,0}$F_{\sigma^{*}_{i-1}}$}%
}}}}
\end{picture}%